\documentclass[11pt]{amsart}

\usepackage[print,nobrazil,nokeys]{vborges} %print: for white background, nokeys: for normal
\usepackage{enumerate}
\usepackage[shortlabels]{enumitem} % For upright item numbering inside theorems

\usepackage{amsmath,dsfont,amsfonts,amssymb,amsxtra,latexsym,amscd,enumerate,amsthm,xcolor}
\usepackage{fullpage}
\usepackage{dsfont} % for bold 1

\usepackage{subcaption}
\usepackage{import}
\usepackage{xifthen}
\usepackage{pdfpages}
\usepackage{transparent}

\usepackage{booktabs} % for fancy tabulars
\usepackage{multirow} % for fancy tabulars

\usepackage{csquotes}

\usetikzlibrary{calc}

\newenvironment{nalign}{
    \begin{equation}
    \begin{aligned}
}{
    \end{aligned}
    \end{equation}
    \ignorespacesafterend
}

\definecolor{cite_color}{HTML}{00629B}
\definecolor{link_color}{HTML}{182B49}

\hypersetup{
    colorlinks=true,
    linkcolor=link_color,
    filecolor=cite_color,      
    urlcolor=cite_color,
    citecolor=cite_color,
    linkbordercolor=cite_color,
    urlbordercolor=cite_color,
    citebordercolor=cite_color,
 }
\usepackage{mathtools}
\newcommand{\defeq}{\vcentcolon=}

\DeclarePairedDelimiter\norm{\lVert}{\rVert}
\DeclarePairedDelimiter\inner{\langle}{\rangle}

\makeatletter
\let\oldnorm\norm
\def\norm{\@ifstar{\oldnorm}{\oldnorm*}}
\let\oldinner\inner
\def\inner{\@ifstar{\oldinner}{\oldinner*}}
\makeatother

\renewcommand{\epsilon}{\varepsilon}
\newcommand{\1}{\mathds{1}}
\DeclareMathOperator{\cl}{\mathrm{cl}}

\DeclareMathOperator{\med}{med}

\newcommand{\s}{\Delta}
\newcommand{\kg}{\langle D\rangle}

\newcommand{\on}{\operatorname}

\usepackage[backend=biber,sorting=nyt,style=alphabetic]{biblatex}
\title{Scattering for the Klein-Gordon-Schrödinger system in three dimensions with radial data}

\author{Vitor Borges}
\address[V. Borges]{Department of Mathematics, University of California, San Diego, La Jolla 92093, USA}
\email{vborges@ucsd.edu}

\author{Tiklung Chan}
\address[T.~Chan]{Department of Mathematics, University of California, San Diego, La Jolla 92093, USA}
\email{tic017@ucsd.edu}

\date{}

\begin{document}

\begin{abstract}
We prove global well-posedness and scattering for the 3D Klein-Gordon-Schrödinger system for small
radial data in the best known global well-posedness range $(u_0, n_0, n_1)\in L^2\times H^{ -\frac{1}{2} + \epsilon } \times H^{
-\frac{3}{2} +\epsilon }$ for any $ \epsilon > 0 $.
The proof uses a global-in-time iteration scheme in the adapted function spaces $U^2_\Phi$,
radial Strichartz estimates, and bilinear restriction estimates.
\end{abstract}

\maketitle

%\noindent
%\textbf{Keywords:} Schrödinger equation; Klein-Gordon equation; global dynamics initial-value problem; radial initial data; low regularity global wellposedness; $U^2$ and $V^2$ spaces.

\section{Introduction}\label{section 1}
\noindent
The Klein--Gordon--Schrödinger system is a basic model for the interaction of a complex nucleon field $u$
and a real meson field $n$ coupled through the Yukawa interaction.
Mathematically, the system emerges as an effective equation 
in a mean-field scaling of the Nelson model \cite{AF}, and is most commonly formulated as
\begin{align}\label{kgs}
    \begin{cases}
        (i\partial_{t} + \Delta)u = un\\
        (\square +1) n = \pm\lvert u\rvert ^2,
    \end{cases}\tag{KGS}
\end{align} 
where $u:\R^{3+1}\rightarrow\C$ and $n:\R^{3+1}\rightarrow\R$ and
$\square=\partial_t^2-\Delta$ is the wave operator. 

We are interested in the long-time behavior of solutions to
(\ref{kgs}) arising from initial data
\begin{align}\label{eq:initial}
  u_0\in H^{ s } (\R^{ 3 } ),
  \quad
  (n_0,n_1)\in H^{ r }(\R^{ 3 } )\times H^{ r-1 } (\R^{ 3 } ).
\end{align} 

%We now review the current state of the three dimensional Klein–Gordon-Schrödinger system.

Local well-posedness in low regularity was obtained by Pecher, first for $ s = r = 0$ \cite{Pecher-kgslwp} and later improved to the near optimal range \cite{Pecher-kgslwp-sharp}
\begin{align*}
  s > -\frac{1}{4},\quad r > -\frac{1}{2},\quad  r - 3s < \frac{3}{2},\quad r-2 < s < r+1
\end{align*}
by constructing a contraction in an $ X^{ s,b }$ space.
On the global side, Colliander–Holmer–Tzirakis \cite{CHT} proved global well-posedness at 
$s = r = 0$ by combining local theory with conservation of mass, and Pecher \cite{Pecher-new} subsequently obtained the best known large data global well-posedness result for
\begin{align*}
    s\geq 0,
    \quad
    s-\frac{1}{2} < r < s+\frac{3}{2}
\end{align*}
building on the works of Bejenaru, Herr, Holmer and Tataru \cite{BHHT} and Bejenaru and Herr
\cite{BH-convolutions} on the closely related Zakharov system.
These global results, however, rely on conservation of mass and therefore do not provide asymptotic information on the solution.

We say that a global solution $ (u,n) $ of (\ref{kgs}) with initial data $ (u,n,\partial_{ t
} n)|_{ t=0 } = (u_0,n_0,n_1) $ \textit{scatters} in $ H^{ s } \times H^{ r }   $ if there
exists a solution to the free equations
\begin{align*}
    \begin{cases}
        (i\partial_{t} + \Delta)u^{ * }  = 0 \\
        (\square +1) n ^{ * }  = 0,
    \end{cases}
\end{align*}  
such that
\begin{align*}
  \lVert u(t) - u^{ * } (t) \rVert  _{ H^{ s } _{ x }  } + \lVert n(t) - n ^{ * } (t) \rVert
  _{ H^{ r } _{ x }  }  + \lVert \partial_{ t } n(t) - \partial_{ t } n ^{ * } (t)  \rVert
  _{ H^{ r-1 } _{ x }  } 
  \longrightarrow 0
\end{align*} 
as $ t\to \infty $.

The first scattering result for (\ref{kgs}) was obtained by Banquet, Ferreira and
Villamizar-Roa \cite{BFV} for initial data in weak $ L^{ r }$ spaces, an  infinite
$ L^{ 2 }  $-norm setting.
More recently, You \cite{You} obtained scattering for small, localized, and highly regular initial data on Sobolev spaces using the method of spacetime resonances.
By contrast, in the non-localized setting, the asymptotic behavior of solutions remains largely open even for small initial data.

In this paper, we present the first approach to low regularity global well-posedness for
(\ref{kgs}) for small radial data without relying on energy conservation. This is the first
step towards understanding global asymptotics for non-localized initial data.  
We construct a global in time iteration scheme in carefully designed resolution spaces and,
as a byproduct, obtain linear scattering in the best known global well-posedness range $ s =
0$ and $ r > -\frac{1}{2}  $.
Our main result is the following:

\begin{thm}\label{mainthm}
  Let  $ s\geq 0 $ and let 
  $ \left( s-\frac{1}{2} \right) < r < (s+2)$.
  For all sufficiently small radial initial data
  \begin{align*}
    u_0\in H^{s} (\R^{ 3 },\C),
    \quad
    (n_0,n_1)\in H^{ r } (\R^{ 3 },\R) \times H^{ -1+r }  (\R^{ 3 },\R ),
  \end{align*} 
  the Cauchy problem (\ref{kgs}) admits a global solution
  \begin{align}\label{eq:solclass}
    u \in C_{ t }\left( \R; H_{ x }^{ s }(\R^{ 3 },\C )  \right),
    \quad
    n \in  C_{ t }\left( \R; H_{ x }^{ r }(\R^{ 3 },\R )  \right) \cap
    C_{ t }^{ 1}\left( \R; H_{ x }^{ r-1 }(\R^{ 3 },\R )  \right)
  \end{align} 
  which scatters in $ H_{ x } ^{ s }\times H_{ x } ^{ r } $ to free solutions as $ t\to \pm
  \infty  $.
  Moreover, the solution is unique within the resolution spaces constructed in Subsection
  \ref{resolution}.
\end{thm}

\subsection{Main difficulties and ideas of the proof}

We now discuss the difficulties in proving Theorem \ref{mainthm} and the main ideas of the proof. We begin by highlighting the features of the Klein-Gordon-Schrödinger system in contrast to the Zakharov system; we directly compare our approach to that of prior works on the Zakharov system under spherical symmetry \cite{GN,KK}.
For the sake of
exposition, we deliberately overlook some technical aspects in favor of a cleaner heuristic presentation.
We focus on the case $ s = 0 $.

Schematically, the main difficulty in proving global well-posedness for (\ref{kgs}) is to establish trilinear estimates of the form
\begin{align}\label{eq:toy_new}
  \left\lvert \int_{ \R^{ 1+3 }  } n u_{1} u_2 \dif x\dif t \right\rvert 
  &\lesssim 2^{ rk_+ } \lVert n \rVert_{ S_{ \kg }  }   \lVert u_1 \rVert_{ S_{\s} }   
    \lVert u_2 \rVert_{ S_{ \s }  },
\end{align}
where $k_+=\max\{k,0\}$, and
\begin{align*}
  n = e^{ it\langle \nabla \rangle  } P_{ k } g,
  \quad
  u_1 = e^{ -it\Delta } P_{ k_1 } f_1,
  \quad\text{and}\quad 
  u_2 = e^{ -it\Delta } P_{ k_2 } f_2
\end{align*} 
are free waves localized to frequencies $ 2^{ k }$ , $2^{ k_1 } $, and $ 2^{ k_2 }  $.
Part of the difficulty is in constructing solution spaces $ S_{ \s }  $ and $ S_{ \kg }  $
adapted to the Schrödinger and Klein-Gordon evolutions in which such estimates can be proved.

A natural first approach is to build $S_{\s}$ and $S_{\kg}$ based on the Strichartz estimates 
\begin{align*}
  \lVert e^{ -it\Delta } f \rVert  _{  L_{ t }^{ 2 } L_{ x }^{ 6 }   }   
  \lesssim \lVert f \rVert  _{ L_{ x } ^{ 2 }  }
  \quad\text{and}\quad 
  \lVert e^{ it\langle \nabla \rangle  } P_{ k } g \rVert  _{ L_{ t }^{ \infty } L_{ x }^{ 2 }   } 
  \lesssim \lVert g \rVert  _{ L^{ 2 } _{ x }  }.
\end{align*} 
However, Hölder's inequality only gives
\begin{align*}
  \lVert n u_1u_2 \rVert_{ L_{ t }^{ 1 } L_{ x }^{ \frac{6}{5}  }   }   
  &\lesssim 2^{ \frac{k}{2}  } \lVert n \rVert_{ S_{ \kg }  } \lVert u_1 \rVert_{ S_{ \s } } \lVert u_2 \rVert
  _{ S_{ \s }   }, 
\end{align*} 
which is not enough for (\ref{eq:toy_new}).
The restrictive Strichartz range is a manifestation of the fact that waves disperse more
slowly as one goes down in space dimension. 
This is a recurring difficulty in most nonlinear models in low dimensions.

Under radial symmetry, one gains access to a substantially larger range of Strichartz estimates; spherical symmetry prevents concentration along small angular regions and
rules out certain Knapp-type constructions.
In particular, ignoring endpoint issues, one has
\begin{align}\label{eq:strichartz-toy}
  2^{ \frac{k_{ 1 } }{4}  } \lVert e^{ -it\Delta } P_{ k_1 } f \rVert_{ L_{ t }^{ 2 } L_{ x
  }^{ 4 } }  
  \lesssim \lVert f \rVert  _{ L_{ x } ^{ 2 }  }
  \quad\text{and}\quad 
  2^{ -\frac{k}{4}   }
  \lVert e^{ it\lvert \nabla \rvert   } P_{ k } h \rVert  _{ L_{ t }^{ 2} L_{ x }^{ 4 }   } 
  \lesssim \lVert h \rVert  _{ L^{ 2 } _{ x }  }.
\end{align} 
In \cite{GN}, Guo and Nakanishi show that the increased radial Strichartz range is enough
to prove global well-posedness and scattering for radial solutions to the three-dimensional
Zakharov system in the energy space $ H^{ 1 } \times L^{ 2 } \times \dot H^{ -1 }  $.
Roughly speaking, their approach begins by separating nonlinear interactions into resonant and nonresonant
contributions.
The nonresonant components are addressed via a normal form transformation that replaces the
quadratic nonlinearity for a better-behaved cubic one.
The resonant interactions are then controlled by the \textit{heuristic} estimate
\begin{nalign}\label{eq:multizakharov}
  &\left\lvert \int_{ \R^{ 1+3 }  } 
    e^{ it\lvert \nabla \rvert  } P_{ k } h
    \cdot e^{ -it\Delta } P_{ k_1 } f_1
    \cdot e^{ -it\Delta } P_{ k_2 } f_2 \dif x\dif t \right\rvert \\
  &\hspace{15em}
    \lesssim 
    \lVert e^{  t\lvert \nabla \rvert } P_{ k } h \rVert_{ L_{ t }^{ 2 } L_{ x }^{ 4 }    }   
    \lVert e^{ -it\Delta } P_{ k_1 } f_1 \rVert_{ L_{ t }^{ 2 } L_{ x }^{ 4 } } 
    \lVert e^{ -it\Delta } P_{ k_2 } f_2 \rVert_{ L_{ t }^{ \infty } L_{ x }^{ 2 }  }\\
  &\hspace{15em} 
    \lesssim 2^{ \frac{ (k-k_1) }{ 4 }  } 
    \lVert e^{  it\lvert \nabla \rvert } P_{ k } h \rVert_{ S_{ \lvert D \rvert   } }   
    \lVert e^{ -it\Delta } P_{ k_1 } f_1 \rVert_{ S_{ \s }  } 
    \lVert e^{ -it\Delta } P_{ k_2 } f_2 \rVert_{ S_{ \s } }.
\end{nalign} 
In the regime where $ k_1 \succ k $, this corresponds to global well-posedness
with $ r = 0 $ and, ignoring issues of summability, is enough to close the argument in the energy space.

More recently, Kato and Kinoshita \cite{KK} obtained global well-posedness
and scattering for the Zakharov system for small radial data in the critical space $ L^{ 2 } \times H^{ -\frac{1}{2} } \times H^{
-\frac{3}{2}  }  $.
Compared to \cite{GN}, their resonance analysis takes advantage of the discrepancy in derivative gains between the Strichartz estimates for the
Schrödinger and the wave equations in (\ref{eq:strichartz-toy}), see Figure
\ref{fig:Strichartz} for the full Strichartz ranges.
%For instance, when $ k_1\succ k \geq 1 $, we can estimate
%\begin{align*}
%  &\left\lvert \int_{ \R^{ 1+3 }  } 
%    e^{ it\lvert \nabla \rvert  } P_{ k } h
%    \cdot e^{ -it\Delta } P_{ k_1 } f_1
%    \cdot e^{ -it\Delta } P_{ k_2 } f_2 \dif x\dif t \right\rvert \\
%  &\hspace{15em}
%    \lesssim \lVert e^{ -it\lvert \nabla \rvert } P_{ k } h \rVert_{ L_{ t }^{ \infty } L_{
%    x }^{ 2 } }   
%    \lVert e^{ -it\lvert \nabla \rvert } P_{ k_1 } f_1 \rVert_{ L_{ t }^{ 2 } L_{ x }^{ 4 } } 
%    \lVert e^{ -it\lvert \nabla \rvert } P_{ k_2 } f_2 \rVert_{ L_{ t }^{ 2 } L_{ x }^{ 4}}\\
%  &\hspace{15em} 
%    \lesssim 2^{ \frac{ (-k_1-k_2) }{ 4 }  } 
%    \lVert e^{ -it\lvert \nabla \rvert } P_{ k } h \rVert_{ S_{ \lvert D \rvert   } }   
%    \lVert e^{ -it\lvert \nabla \rvert } P_{ k_1 } f_1 \rVert_{ S_{ \s }  } 
%    \lVert e^{ -it\lvert \nabla \rvert } P_{ k_2 } f_2 \rVert_{ S_{ \s } }.
%\end{align*} 
In the nonresonant regime, they forego the normal form transformation in favor of the modulation spaces $ \dot X^{ s,b } $,
which allows finer control over nonresonant interactions by tracking the distance
between bilinear interactions and the characteristic surfaces $ \tau =  \lvert \xi \rvert ^2 $ and $
\tau =  \lvert \xi \rvert  $ in Fourier space.
The fine $\dot X^{s,b}$ control also requires a delicate case-by-case resonance analysis.

The Klein-Gordon-Schrödinger system differs from the Zakharov system in a key aspect. 
Unlike the Schrödinger and wave propagators, which maintain the same qualitative behavior across different scales, the Klein-Gordon propagator behaves differently at low and high frequencies.
More precisely, under spherical symmetry,
\begin{align*}
  2^{ -\frac{k}{4} } \lVert e^{ it\langle \nabla \rangle  } P_{ k } g \rVert  _{ L_{ t }^{ 2
    } L_{ x }^{ 4 }  }
  &\lesssim \lVert g \rVert  _{ L^{ 2 }_{ x }   },\,
    \text{ if } k \geq 0,
\end{align*}
whereas
\begin{align*}
    2^{ \frac{k}{4} } \lVert e^{ it\langle \nabla \rangle }P_{ k } g   \rVert  _{ L_{ t }^{
    2 } L_{ x }^{ 4 }  }
  &\lesssim \lVert g \rVert  _{ L^{ 2 }_{ x }   },
    \, \text{ if } k < 0.
\end{align*} 
This matches the geometry of the characteristic surface $\tau = \langle \xi \rangle $, which is parabolic near the origin and asymptotically conic at high frequency.
In particular, at low frequency, the
Klein–Gordon evolution is of Schrödinger type, rather than wave type.
As a consequence, when $ k < 0 $, the analogue of (\ref{eq:multizakharov}) becomes
\begin{align}\label{eq:failure}
  &\left\lvert \int_{ \R^{ 1+3 }  } 
    e^{ it\langle \nabla \rangle   } P_{ k } g
    \cdot e^{ -it\Delta } P_{ k_1 } f_1
    \cdot e^{ -it\Delta } P_{ k_2 } f_2 \dif x\dif t \right\rvert \\
  &\hspace{14.5em} 
    \lesssim 2^{ \frac{ (-k-k_1) }{ 4 }  } 
    \lVert e^{  it\langle \nabla \rangle  } P_{ k } g \rVert_{ S_{ \lvert D \rvert   } }   
    \lVert e^{ -it\Delta } P_{ k_1 } f_1 \rVert_{ S_{ \s }  } 
    \lVert e^{ -it\Delta } P_{ k_2 } f_2 \rVert_{ S_{ \s } }\nonumber,
\end{align} 
which is unbounded as $ k\to -\infty $.
This low frequency behavior creates the main obstruction in our analysis. In the regime $k_1 \succ k$
and $k < 0$, the interaction is resonant, and the estimate above is far from sufficient.
Moreover, unlike some other quadratic Klein-Gordon models in the literature, such as \cite{BH-dkg, BB}, we lack a null structure to exploit in the near-resonant regions. 
The presence of resonances also prevents the use of normal form transformations or $\dot X^{s,b}$ spaces, so the approaches in \cite{GN,KK} do not apply in this regime.
To control these interactions, we use bilinear restriction estimates instead. 
In the problematic frequency regime, the Fourier
supports of $ e^{ it\langle \nabla \rangle  } P_{ k } f $ and $ e^{ -it\Delta } P_{ k_1 }
g_1 $ are \emph{transversal}.
Heuristically, waves traveling in transversal directions interact for shorter times
which lead to bilinear gains that are not visible at the linear level.
For instance, one can improve (\ref{eq:failure}) by estimating
\begin{align*}
  &\left\lvert \int_{ \R^{ 1+3 }  } 
    e^{ it\langle \nabla \rangle   } P_{ k } g
    \cdot e^{ -it\Delta } P_{ k_1 } f_1
    \cdot e^{ -it\Delta } P_{ k_2 } f_2 \dif x\dif t \right\rvert \\
  &\hspace{14em} 
    \lesssim 
    \lVert e^{  it\langle \nabla \rangle  } P_{ k } g\cdot e^{ -it\Delta } P_{ k_1 } f_1 \rVert_{
    L_{ t }^{ \frac{8}{5} } L_{ x }^{ \frac{3}{2}  }} \rVert
    \lVert e^{ -it\Delta } P_{ k_2 } f_2 \rVert_{ L_{ t }^{ \frac{8}{3}  } L_{ x }^{ 3 } }\\
  &\hspace{14em} 
    \lesssim 2^{ \frac{k}{12} - \frac{7}{12}k_1  } 
    \lVert e^{  it\langle \nabla \rangle  } P_{ k } g \rVert_{ L_{ t }^{ \infty } L_{ x }^{2} }   
    \lVert e^{ -it\Delta } P_{ k_1 } f_1 \rVert_{ L_{ t }^{ \infty } L_{ x }^{2} }   
    \lVert e^{ -it\Delta } P_{ k_2 } f_2 \rVert_{ S_{ \s } },
\end{align*} 
which is bounded in the region  $ k_1\sim k_2 \geq 1$, $ k_1\succ k $, and compensates for the failure in the linear Strichartz argument. Note that the exponents $L_t^{\frac{8}{5}}L_x^{\frac{3}{2}}$ for the bilinear estimate, while not optimal, suffice for our nonlinear analysis, see Theorem \ref{thm:bilinearrange} or \cite[Corollary 1.6]{Ca}.

There is a rich history of bilinear restriction estimates in harmonic
analysis  -- see for example \cite{Wolff, Tao1, Tao2, LeeVargas, Bejenaru} and the
references therein. 
Bilinear restriction ideas have also appeared
earlier in nonlinear dispersive equations, especially in $L^2$-based
nonresonant settings, see for instance \cite{KS,BHHT,BH-convolutions},
but the use of bilinear restriction estimates to study resonant interactions is more recent, see \cite{Ca,CH}.
While our analysis so far has been formulated in terms of free waves, the full nonlinear analysis requires a mechanism to transfer bilinear estimates from free waves to higher order nonlinear iterates, which has historically been a major obstacle in applications to PDEs.
In this paper, we rely on the framework developed by Candy \cite{Ca}, which builds on the $U^2$ and $V^2$ spaces introduced by Tataru in an unpublished manuscript and then further developed by Koch and Tataru \cite{KT-u2-1,
KT-u2-2}. 
This point of view has already proved useful in important nonlinear dispersive PDEs, for instance in the works \cite{CH,CH-wm}; however its implementation for systems with mixed phases has been more limited, see \cite{CHK}.

Intuitively, $U^2_{\Phi}$ is an atomic space built out of free solutions, and therefore
inherits both the linear estimates and the bilinear restriction bounds available at the free level. For the purposes of this work, we combine radial Strichartz estimates, modulation bounds in $\dot X^{s,b}$ spaces, and bilinear
restriction estimates to construct a contraction in $U^2_{\Phi}$-based resolution spaces. The companion space $V^2_{\Phi}$ appears naturally through the duality argument for the Duhamel terms and creates some
additional technical difficulties, especially near endpoint exponents.
The missing regularity endpoint $(s,r) = (0,-1/2)$ in Theorem \ref{mainthm} is directly related to the failure of certain endpoint Strichartz estimates and to the space $V^2$, as discussed in Subsection 4.3.

The proof of the main trilinear estimates is ultimately organized as a case analysis according to the relative sizes of $2^k$, $2^{k_1}$, and $2^{k_2}$, see Figures \ref{fig:schrodinger} and \ref{fig:kleingordon} for a visual breakdown of the cases.
Compared to \cite{KK}, the bilinear restriction estimates simplify and shorten the resonance analysis considerably. 
Once the frequency localized trilinear estimates are established, the global iteration scheme and the scattering statement follow from a standard contraction argument in $U^2$--$V^2$.

\subsection{Outline of the paper}
In Section 2 we introduce some of the basic notation and rewrite the original system (\ref{kgs}) in the equivalent first-order form (\ref{eq:modifiedkgs}).
The resonance analysis for the Schrödinger and Klein-Gordon nonlinearities is done in Subsection 2.2.
In Section 3 we describe the linear theory for the propagators $e^{-i\Delta}$ and $e^{it\langle \nabla\rangle}$ under spherical symmetry.
We also define the spaces $U^2$ and $V^2$, and construct the resolution spaces that will be used in the contraction.
Finally, in Section 4, we prove the multilinear estimates required for the iteration argument, organized into several steps. In Subsection 4.1 we establish the bilinear restriction estimates in the relevant frequency regions and check the necessary transversality conditions. 
Subsection 4.2 contains a high-level overview and a visual breakdown of the case analysis needed in the proof of the trilinear estimates, whereas Subsections 4.3 and 4.4 contain the main nonlinear analysis and comprise the technical body of the paper.
In Subsection 4.5 we provide a short proof of the main theorem (\ref{mainthm}) assuming the frequency localized bounds.

\section{Preliminaries}\label{section 2}
\subsection{Notation}

We say $A\lesssim B$ (analogously $A\gtrsim B$) if there is a constant $C>0$ such that
$A\leq C B$ and we say $A\sim B$ if both $A\lesssim B$ and $B\lesssim A$. For integers $j,k$
we say $j\prec k$ (analogously $j\succ k$) if $j\leq k-10$.
Given integers $ k,k_1,k_2 $ we denote the median of $ k,k_1,k_2 $ by $ \med\left\{ k,k_1,k_2 \right\} $.
In addition, given $ \xi\in\R^{ n } $, we define the Japanese bracket of $ \xi $
by $ \langle \xi \rangle \defeq \sqrt{ 1+\lvert \xi \rvert ^2 }  $.

Let $\rho^0\in C^\infty_c(-2,2)$ be a fixed smooth, even, cutoff
satisfying $\rho^0(s)=1$ for $|s|\leq 1$ and $0\leq \rho\leq 1$. For
$k \in \Z$ we define $\rho_k:\R \ \mbox{or} \ \R^2 \to \R$,
$\rho_k(y)\defeq\rho^0(2^{-k}|y|)-\rho^0(2^{-k+1}|y|)$.
Let $\tilde{\rho}_k=\rho_{k-1}+\rho_k+\rho_{k+1}$. For $k \geq 1$, let $P_k$ be the
Fourier multiplication operators with respect to $\rho_k$, and
$P_0=I-\sum_{k \geq 1}P_k$.  For $j \in \Z$ we define
\begin{align*}
    \mathcal{F}[Q^{\Delta}_{j}f](\tau,\xi)\defeq \rho_j(\tau-\lvert \xi\rvert^2)\mathcal{F}f(\tau,\xi),\\
    \mathcal{F}[Q^{\pm \langle D\rangle }_{j}f](\tau,\xi)\defeq \rho_j(\tau\pm \langle \xi\rangle)\mathcal{F}f(\tau,\xi).
\end{align*}
Similarly, we define $\tilde P_k, \tilde Q_j^{\Delta}, \tilde Q_j^{\pm\langle D\rangle}$, where $D = -i\nabla$. 
We also define $P_{\leq k} \defeq \sum_{k'\leq k} P_{k'}$, $\tilde P_{\leq k}, \tilde
Q_{\leq j}^{\Delta}$, and $ \tilde Q_{\leq j}^{\pm\langle D\rangle}$ analogously. 

%Let $ s\in\R $ and $ 1\leq p,q\leq \infty$, we define the Besov space $ B_{ p,q }^{ s }  $
%as the completion of the space of smooth and compactly supported test functions with norm
%\begin{align*}
%  \lVert u \rVert  _{ B_{ p,q }^{ s }  } 
%  &\defeq \lVert P_{ \leq 0 } u \rVert  _{ L^{ p }  }  + \left( \sum_{ k\geq 1 } 2^{ qks }
%  \lVert P_{ k } u \rVert_{ L^{ p } }^{ q }    \right) ^{ \frac{1}{q}  } 
%\end{align*} 
%with the usual modification when $ q = \infty  $.
%One can directly check from the definition above that  
%\begin{align*}
%  B_{ p,1 }^{ s } \subset  W^{ s,p } \subset  B_{ p,\infty  }^{ s },
%\end{align*} 
%where $ W^{ s,p }  $ denotes the standard $ L^{ p }  $-based Sobolev space at regularity $ s
%$.
%When $ p = 2 $, we write $ H^{ s } \defeq W^{ s,2 } $ and the embedding reads
%\begin{align*}
%  B_{ 2,1 }^{ s } \subset  H^{ s } = B_{ 2,2 }^{ s }  \subset  B_{ 2,\infty  }^{ s }.
%\end{align*} 

\subsection{Setup of the system and resonance analysis}

By letting $N \defeq (1-i\langle D\rangle ^{-1}\partial_{t})n$, we can rewrite (\ref{kgs}) as the first order system
\begin{nalign}\label{eq:modifiedkgs}
  &(i\partial_{t}+\Delta)u =  \frac{1}{2} u(N+\overline N)\\
   &(i\partial_{t}+\langle D\rangle)N = \pm\langle D\rangle ^{-1}\lvert u\rvert ^2.
\end{nalign}
We aim to provide a global theory for (\ref{eq:modifiedkgs})
with initial data $(u,N)|_{t=0} = (u_0, N_0) \in H^s \times H^r$ with $s\geq 0$ and $(s-\frac 12) < r < (s+2)$.
This easily translates back to the global theory for (\ref{kgs}) with initial data
$(u_0,n_0,n_1)\in H^s \times H^r \times H^{r-1}$.
Furthermore, for small data, the sign for the Klein-Gordon nonlinearity does not affect the argument, so we fix it to $+$.

Define the resonance functions corresponding to the Schrödinger and Klein-Gordon nonlinearities $uN$, $u\overline N$, and $\langle D\rangle ^{-1} |u|^2$, respectively:
\begin{align}
    \Phi_{+} (\xi,\eta) &= \lvert \xi\rvert^2 - \langle\eta\rangle - \lvert
    \xi-\eta\rvert^2,\\
    \Phi_{-}(\xi,\eta) &= \lvert \xi\rvert^2 + \langle\eta\rangle -\lvert
    \xi-\eta\rvert^2,\\
    \Psi(\xi,\eta) &= \langle\xi\rangle -\lvert \xi-\eta\rvert^2 + \lvert \eta\rvert^2.
\end{align}

We make note of some completely nonresonant interactions in the following lemmas whose proofs follow immediately from direct computations:
\begin{lemma}[Schrödinger nonresonances]\label{lemma:sch nonresonance}
Let $\abs{\xi}\sim 2^{k_1}$, $\abs{\eta}\sim 2^{k}$, and $\abs{\xi-\eta}\sim 2^{k_2}$. We then have the following estimates for $\Phi_{\pm}$:
\begin{enumerate}[label=\upshape(\roman*)]
    \item If $k_1, k\leq -3$ then $\abs{\Phi_{\pm}(\xi,\eta)}\geq\frac{1}{2}$,
    \item If $k\succ k_1$ and $k>7$ then $\abs{\Phi_{\pm}(\xi,\eta)}\gtrsim 2^{2k}$.
    \item If $k\sim k_1\succ k_2$ then $\abs{\Phi_{\pm}(\xi,\eta)}\gtrsim 2^{2k}$
\end{enumerate}
\end{lemma}

\begin{lemma}[Klein-Gordon nonresonances]\label{lemma: kg nonresonance}
Let $\abs{\xi}\sim 2^{k_1}$, $\abs{\eta}\sim 2^{k}$, and $\abs{\xi-\eta}\sim 2^{k_2}$. We then have the following estimates for $\Psi$:
\begin{enumerate}[label=\upshape(\roman*)]
    \item If $k_1,k_2\leq -3$ then $\abs{\Psi(\xi,\eta)}\geq\frac{1}{2}$,
    \item If $k_1\succ k_2$ and $k_1>10$ then $\abs{\Psi(\xi,\eta)}\gtrsim 2^{2k_1}$.
\end{enumerate}
Furthermore, the above estimates remain true if we swap the roles of $k_1$ and $k_2$.
\end{lemma}

\section{Resolution spaces}\label{section 3}
In this section, we develop the linear theory for the PDEs
\begin{align*}
  (i\partial_{ t } + \Delta) u &= f, \quad u(0) = u_0,\\
  (i\partial_{ t } - \langle D \rangle ) N &= g, \quad N(0) = N_0,
\end{align*} 
which are the basic models for (\ref{eq:modifiedkgs}), and construct the resolution spaces
we will use in our contraction argument.

\subsection{Linear theory}
We start by recalling the frequency localized Strichartz estimates for the homogeneous
Schrödinger and Klein-Gordon propagators under radial symmetry.
Since the Schrödinger evolution is scale-invariant, the Schrödinger Strichartz estimates
behave the same in all scales.
The Klein-Gordon evolution, on the other hand, behaves as Schrödinger in low frequency
regimes and as wave in high frequency. 
These behaviors are quantified in the Strichartz estimates below.
\begin{prop}[Radial Strichartz estimates for Schrödinger and Klein-Gordon]\label{linearstructure}
  Let $(p_{ S } ,q_{ S } )$ and $ (p_{ W } ,
  q_{ W } )$ be radially admissible Schrödinger and wave Strichartz pairs,
  i.e.,
  \begin{nalign}\label{eq:ranges}
    &2\leq p_{ S }, q_{ S } \leq \infty,\quad 
    \frac{2}{p_{ S } } + \frac{5}{q_{ S } }\leq \frac{5}{2},\quad
    (p_{ S } ,q_{ S } )\neq \left(2,\frac{10}{3}\right),\\
    &2\leq p_{ W }, q_{ W } \leq \infty,\quad 
    \frac{1}{p_{ W } } + \frac{2}{q_{ W } }\leq 1,\quad
    (p_{ W } ,q_{ W } )\neq \left(2,4\right),
  \end{nalign} 
  and let $ f,g\in L_{ x } ^{ 2 }  $ be radial functions.
  Then,
  \begin{enumerate}[label = \upshape(\roman*)]
    \item For any $k \in \Z$,
      \begin{align}\label{eq:StrichartzS}
        2^{ k\sigma_{ S }} 
        \lVert P_{ k } e^{ -it\Delta }f  \rVert _{ L_{ t }^{ p_{ S }  } L_{ x }^{ q_{ S }  }  } 
        &\lesssim \lVert f \rVert  _{ L_{ x }^{ 2 }  },
        \qquad
        \sigma_{ S }(p,q) \defeq \frac{2}{p} + \frac{3}{q} - \frac{3}{2}.
      \end{align} 

    \item Moreover, for any $ k\geq 0 $,
      \begin{align}\label{eq:StrichartzW}
        2^{ k\sigma_{ W }} 
        \lVert P_{ k } e^{ it\langle D \rangle   }g  \rVert _{ L_{ t }^{ p_{ W }  } L_{ x }^{ q_{ W }  }  } 
        &\lesssim \lVert g \rVert  _{ L_{ x }^{ 2 }  },
        \qquad
        \sigma_{ W }(p,q) \defeq \frac{1}{p} + \frac{3}{q} - \frac{3}{2},
      \end{align} 
      and for  any fixed frequency cutoff $ c $, 
      \begin{align*}
        2^{ k\sigma_{ S }} 
        \lVert P_{ k } e^{ it\langle D \rangle   }g  \rVert _{ L_{ t }^{ p_{ S }  } L_{ x }^{ q_{ S }  }  } 
        &\lesssim \lVert g \rVert  _{ L_{ x }^{ 2 }  }
      \end{align*} 
      uniformly in $ k\leq c $.
  \end{enumerate}
  Furthermore, the admissible ranges above are sharp.
\end{prop}
This result, or equivalent formulations thereof, can be found in many places in the literature, for example \cite{GN, GNW}. In fact, since the curvature of the Klein-Gordon characteristic surface $\xi\mapsto \langle
\xi \rangle  $ is non-degenerate at each localized frequency scale, we can further extend the range of the 
Klein-Gordon admissible pairs from the wave region to the Schrödinger region, albeit
incurring extra derivative losses in high frequency, see \cite[Lemma
3.1]{GNW}. 
For the purposes of this paper, the wave admissible range (\ref{eq:ranges}) is sufficient in the high frequency regime.

For the convenience of the reader, we record in Table \ref{table} the values of $ \sigma_{ S }(p,q) $ and $
\sigma_{ W }(p,q) $ that will be used in Section \ref{section 4}.

\begin{table}[h]
\centering
\makebox[0pt][c]{\parbox{1\textwidth}{%
    \centering
    \begin{minipage}[b]{0.3\textwidth}\centering
      \begin{tabular}{cc}
        $ (p,q) $ & $ \sigma_{ S } (p,q) $  \\
        \midrule[\heavyrulewidth]
        \addlinespace[5pt]
        $ \left( \infty, 2 \right)$ &  $ 0 $  \\
        \addlinespace[5pt]
        $ \left( 4 , 4\right)$ &  $ -\frac{1}{4}  $  \\
        \addlinespace[5pt]
        $ \left( \frac{8}{3} ,3 \right)  $ & $ \frac{1}{4}  $  \\
        \addlinespace[5pt]
        $ (2,6) $ & $ 0  $  \\
        \addlinespace[5pt]
        $ (2,5) $ & $ \frac{1}{10}  $  \\
        \addlinespace[5pt]
        $ (2,4) $ & $ \frac{1}{4}  $  \\
        \addlinespace[5pt]
        $ \left( 2,\frac{10}{3}  \right)  $ & $ \frac{2}{5}  $  \\
        \addlinespace[5pt]
  %      \bottomrule \addlinespace[\belowrulesep]
      \end{tabular}
    \end{minipage}
    \begin{minipage}[b]{0.3\textwidth}\centering
      \begin{tabular}{cc}
        $ (p,q) $ & $ \sigma_{ W } (p,q) $  \\
        \midrule[\heavyrulewidth]
        \addlinespace[5pt]
        $ \left( \infty, 2 \right)$ &  $ 0 $  \\
        \addlinespace[5pt]
        $ \left( 4, 4 \right)$ &  $ -\frac{1}{2}  $  \\
        \addlinespace[5pt]
        $ \left( \frac{8}{3} , 3 \right)$ &  $ -\frac{1}{8}  $  \\
  %      \bottomrule \addlinespace[\belowrulesep]
      \end{tabular}
    \end{minipage}
  }}
  \caption{Reference values for $\sigma_S(p,q)$ and $\sigma_W(p,q)$.}
  \label{table}
\end{table}
Compared to the radial range in (\ref{eq:ranges}), the standard Strichartz range for the
free Schrödinger propagator in three dimensions
\begin{align*}
  &2\leq p_{ S }, q_{ S } \leq \infty,\quad 
  \frac{2}{p_{ S } } + \frac{3}{q_{ S } }  = \frac{3}{2},
\end{align*}
which corresponds to $ \sigma_{ S } (p_{ S } , q_{ S }) = 0$, is much more restrictive.
We use $ \sigma_{ S }  $ to further decompose the set of radially admissible Schrödinger
pairs $ (p,q) $ according to $ \sigma_{ S } (p,q) < 0 $, $\sigma_{ S } (p,q) = 0$, and
$\sigma_{ S } (p,q) > 0$, see Figure \ref{fig:Strichartz}. 
On the scaling line $ \sigma_{ S } = 0 $, the Strichartz estimate (\ref{eq:StrichartzS}) is
uniform across all frequencies, and is classically stated as a single estimate
\begin{align*}
  \lVert e^{ -it\Delta } u_0 \rVert_{ L^{ p } L^{ q } }   \lesssim \lVert u_0 \rVert _{ L^{
  2}  },
\end{align*} 
see \cite[Theorem 2.3]{TaoBook}.
In the region where $ \sigma_{ S } < 0 $, (\ref{eq:StrichartzS}) loses derivatives in high
frequencies; one should think of this region as being accessible from the scaling line
estimates via Sobolev embeddings. 
The extended radial range $ \sigma_{ S } > 0 $ is typically inaccessible for the Schrödinger
evolution outside of spherical symmetry, as can be shown by a Knapp counterexample. 
In this region, (\ref{eq:StrichartzS}) gains derivatives at high frequencies, which will
often be enough to compensate for the loss of derivatives coming from $ \sigma_{ W } $ in
the same region (see Table \ref{table}).

%%%%%%%%% I 

%In low frequencies, (\ref{eq:StrichartzS}) loses derivatives that cannot be compensated by
%$ \sigma_{ W }$, since the Klein-Gordon propagator is also of Schrödinger-type
%in this regime .
%This is a key difference between the Klein-Gordon-Schrödinger system and other similar
%quadratic dispersive systems studied in the literature, such as the Zakharov \cite{GN} or
%Klein-Gordon-Zakharov \cite{GNW}.
%This is also one of the main reasons why our approach to the nonlinear analysis will be split into
%more cases than the aforementioned works.

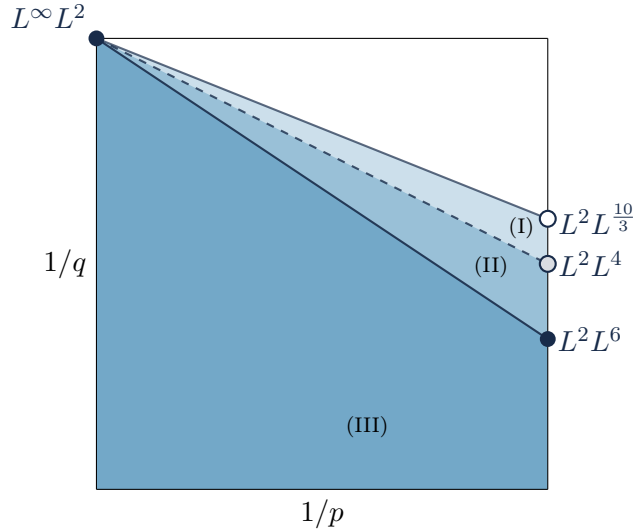
\begin{figure}[hb]
\begin{tikzpicture}[scale=12]
    % ====== COLOR DEFINITIONS (Easy to modify) ======
    % Region colors - using distinct test colors to verify regions
    \definecolor{toplinecolor}{HTML}{182B49}           % Red for top (lowest) line
    \definecolor{middlelinecolor}{HTML}{182B49}        % Orange for middle (highest) line  
    \definecolor{bottomlinecolor}{HTML}{182B49}        % Blue for bottom (dashed) line
    \definecolor{topregion}{HTML}{00629B}              % Light red - between blue and orange lines
    \definecolor{middleregion}{HTML}{00629B}           % Light green - between red and blue lines
    \definecolor{bottomregion}{HTML}{00629B}           % Light blue - between y=0 and red line
    \definecolor{startpointcolor}{HTML}{182B49}        % Green for start point
    \definecolor{axiscolor}{HTML}{000000}              % Black for axes
    % The color can be set to background color to make it transparent
    \definecolor{dualregion}{HTML}{FFFFFF}             % Light yellow - from blue line to top of box
    % ================================================
    % Fill the DUAL region: from the blue dashed line to the top of the box (y=0.5)
    % Blue line: y = (1/2) - (1/2)x
    \fill[dualregion] (0,0.5) -- (0.5,0.25) -- (0.5,0.5) -- (0,0.5) -- cycle;
    
    % Fill the BOTTOM region: between y=0 and the red line
    % Red line: y = (1/2) - (2/3)x
    \fill[bottomregion, opacity=.55] (0,0) -- (0.5,0) -- (0.5,{0.5-2/3*0.5}) -- (0,{0.5-2/3*0}) -- cycle;
    
    % Fill the MIDDLE region: between the red line and blue dashed line
    % Red line: y = (1/2) - (2/3)x
    % Blue line: y = (1/2) - (1/2)x
    \fill[middleregion, opacity=0.40] (0,{0.5-2/3*0}) -- (0.5,{0.5-2/3*0.5}) -- (0.5,0.25) -- (0,0.5) -- cycle;
    
    % Fill the TOP region: between the blue dashed line and orange line
    % Blue line: y = (1/2) - (1/2)x
    % Orange line: y = (1/2) - (2/5)x
    \fill[topregion, opacity=0.20] (0,0.5) -- (0.5,0.25) -- (0.5,0.3) -- (0,0.5) -- cycle;
    
    % Draw the complete border of the square [0, 1/2] x [0, 1/2]
    \draw[axiscolor] (0,0) -- (0.5,0) node[midway, below] {$1/p$};      % Bottom with label
    \draw[axiscolor] (0,0) -- (0,0.5) node[midway, left] {$1/q$};       % Left with label
    \draw[axiscolor] (0.5,0) -- (0.5,0.5);  % Right
    \draw[axiscolor] (0,0.5) -- (0.5,0.5);  % Top
    
    % Draw the three lines
    % Bottom line (solid): y = (1/2) - (2/3)x
    \draw[bottomlinecolor, thick, opacity=0.9] (0,0.5) -- (0.5,{0.5-2/3*0.5});
    
    % Middle line (dashed): y = (1/2) - (1/2)x
    \draw[middlelinecolor, thick, dashed, opacity=0.8] (0,0.5) -- (0.5,0.25);
    % Top line (solid): y = (1/2) - (2/5)x
    \draw[toplinecolor, thick, opacity=0.7] (0,0.5) -- (0.5,0.3);
    
    % Add region labels
    % Define the positioning system using a line passing through all regions
    % Base point for label (III) in bottom region
    \coordinate (base) at (0.3, 0.07);
    % Direction vector for the line (adjust slope here)
    \pgfmathsetmacro{\dx}{0.14}
    \pgfmathsetmacro{\dy}{0.18}
    
    % Label (III) in bottom region - at base point
    \node[font=\scriptsize] at (base) {(III)};
    
    % Label (II) in middle region - along the line from base
    % Parameter t=1 means one unit along the direction vector
    \node[font=\scriptsize] at ($(base) + 0.98*(\dx,\dy)$) {(II)};
    
    % Label (I) in top region - further along the line
    % Parameter t=1.85 (adjust this to move along the line)
    \node[font=\scriptsize] at ($(base) + 1.22*(\dx,\dy)$) {(I)};
    
    % Mark endpoints with dots
    % Top left (all three lines start here)
    \fill[startpointcolor] (0,0.5) circle (0.25pt) node[above left] {$L^\infty L^2$};
    
    % Top line right endpoint (orange, visually highest at x=0.5) - open
    \draw[toplinecolor, fill=white, thick] (0.5,0.3) circle (0.25pt) node[right]
      {$L^2 L^{\frac{10}{3}}$};
    
    % Middle line right endpoint (blue dashed, visually middle at x=0.5) - open
    \draw[middlelinecolor, fill=middlelinecolor!16, thick] (0.5,0.25) circle (0.25pt)
      node[right] {$L^2 L^4$};
    
    % Bottom line right endpoint (red, visually lowest at x=0.5) - filled
    \fill[bottomlinecolor] (0.5,{0.5-2/3*0.5}) circle (0.25pt) node[right] {$L^2 L^6$};
    
\end{tikzpicture}
\caption{Strichartz admissible ranges for the radial wave and Schrödinger equations. 
Region (III) corresponds to $ \sigma_{ S } < 0 $, regions (II) and (III) represent
the increased radial Strichartz range for the wave, and regions (I)--(II) represent the
increased range for Schrödinger.
The faintly colored endpoint $ L^{ 2 } L^{ 4 } $ is included in the Schrödinger admissible range but
not in the wave range.}
\label{fig:Strichartz}
\end{figure}

\subsection{Resolution spaces}\label{resolution}
In this section we define the resolution spaces in which we perform the iteration scheme. 
Let $ 1\leq p < \infty $. 
A function $ \phi :\R\to L_x^2  $ is called a \emph{$ U_{ \s } ^{ p }$-atom}
if there exists a partition $ -\infty = t_0 < t_1 < \cdots < t_{ K }  = \infty $ and
functions $ \left\{ \phi_{ k }  \right\} _{ k=1 }^{ K } \subset L_x^2 $,  $ \phi_1
= 0 $, such that 
\begin{align}\label{eq:atomic}
  \phi(t)
  &= \sum_{ k=1 }^{ K } \1_{ [t_{ k-1 },t_{ k } )}(t) e^{-it\Delta } \phi_{ k },
  \quad
  \sum_{ k=1 }^{ K } \lVert \phi_{ k }  \rVert _{ L_x^2 }^{ p }  = 1.
\end{align} 
Since we are interested in the evolution of radial data, we will assume that atoms are
radial.
We define the atomic spaces $ U_{ \s } ^{ p }$ adapted to the linear Schrödinger
flow by
\begin{align*}
  U_{ \s } ^{ p } 
  &\defeq \left\{ \sum_{ k=1 }^{ \infty  }\lambda_{ k } a_{ k } :
    a_{ k } \text{ is an } U_{ \s } ^{ p }\text{-atom},\, (\lambda_{ k } )_{ k=1 }^{
  \infty } \in \ell^{ 1 } \right\} 
\end{align*} 
with norm
\begin{align*}
  \lVert u \rVert  _{ U_{ \s } ^{ p }  } 
  &\defeq \inf \left\{ \sum_{ k=1 }^{ \infty } \lvert \lambda_{ k } \rvert 
  : u = \sum_{ k=1 }^{ \infty } \lambda_{ k } a_{ k }, a_{ k } \text{ is an } 
  U_{ \s } ^{ p }\text{-atom}\right\}.
\end{align*} 
Similarly, we define the atomic spaces $ U_{ \pm \kg }^{ p }  $ adapted to the (half)
Klein-Gordon flows using atoms
\begin{align}
  \phi(t)
  &= \sum_{ k=1 }^{ K } \1_{ [t_{ k-1 },t_{ k } )}(t) e^{\pm it \langle D \rangle  } \phi_{ k },
  \quad
  \sum_{ k=1 }^{ K } \lVert \phi_{ k }  \rVert _{ L_x^2 }^{ p }  = 1
\end{align} 
and let
\begin{align*}
  U_{ \pm\kg } ^{ p } 
  &\defeq \left\{ \sum_{ k=1 }^{ \infty  }\lambda_{ k } a_{ k } :
    a_{ k } \text{ is an } U_{ \pm\kg } ^{ p }\text{-atom},\, (\lambda_{ k } )_{ k=1 }^{
  \infty } \in \ell^{ 1 } \right\} 
\end{align*} 
equipped with the analogous norm.

As a companion space, we define $ V_{ \s } ^{ p } $ to be the space of right-continuous
functions $ v: \R\to L_x^2 $ such that $ t\mapsto e^{ it\Delta }v(t) $ has bounded
$ p $-variation and $\lim_{t\to\infty} e^{it\Delta}v(t)$ exists in $L_x^2$. 
We equip $V^{p}_\Delta$ with the norm
\begin{align*}
  \lVert v \rVert  _{ V_{ \s } ^{ p }  } 
  &\defeq \sup_{ \left\{ t_{ k }  \right\} _{ k=1 }^{ \infty } \text{ partition} } 
  \left( \sum_{ k=1 }^{ K } \lVert e^{ it_{ k } \Delta } v(t_{ k }) 
    - e^{ it_{ k-1 } \Delta } v( t_{ k-1 } ) \rVert_{ L_{ x }^{ 2 }  } ^{ p }  \right) ^{
  \frac{1}{p}  } < \infty,
\end{align*} 
and similarly for $ V_{ \pm\kg }^{ 2 }  $.
%When $ s = 0 $, we denote $ U_{ \s } ^{ p,s } $ and $ V_{ \s } ^{ p,s }  $ simply by 
%$ U_{ \s} ^{ p } $ and $ V_{ \s } ^{ p} $.

In view of (\ref{eq:atomic}), one should regard $ U^{ 2 }  $ functions as ``almost''
free solutions.
From a technical perspective,  $ U^{ 2 } $  unifies many properties of different 
structures, streamlining both the linear and nonlinear analysis in this paper.
Perhaps most remarkably, $ U^{ 2 } $ also allows for the extension of bilinear restriction
estimates, which typically are only available for free solutions, to higher order iterates. 
We summarize many of the properties of $ U^{ 2 } $ and $ V^{ 2 }  $ in Propositions
\ref{properties}, \ref{inhomogeneous} and \ref{thm:candybilinear}.

For $b\in\R $, we define
\begin{align*}
  \lVert f \rVert _{ \dot X_{ \s } ^{ 0,b,\infty } } 
  \defeq \sup_{ j\in\Z } 2^{ bj } \lVert Q_{ j }^{ \s } f \rVert  _{ L_{ t,x }^{ 2 } } ,
  \quad
  \lVert f \rVert _{ \dot X_{ \kg } ^{ 0,b,
  \infty} } 
  \defeq \sup_{ j\in\Z } 2^{ bj } \lVert Q_{ j }^{ \kg } f \rVert  _{ L_{ t,x }^{ 2 } } .
\end{align*} 

%All constructions above can be easily adapted to the linear Klein-Gordon flows $ e^{ \pm
%i\langle D \rangle  }$, from which we get the spaces $ U_{ \pm\kg  }^{ 2 }
%$, $ V_{ \pm\kg  }^{ 2 } $, and $ \dot X^{ 0,b } _{ \pm\kg }  $.

%For convenience, we state the properties of $U^2$ and $V^2$ for $ s= 0 $ however the same
%properties holds for $ s \neq 0 $ with minor modifications.
%the Schrödinger variants and
%unless otherwise noted, the reader should assume that the same properties hold for the
%Klein-Gordon variants and 
\begin{prop}\label{properties}
  The following hold:
  \begin{enumerate}[(i)]
    \item The operators $ Q_{ j } ^{ \s }$ (resp. $ Q_{ j }^{ \kg }  $ ), $ j\in\Z $,
      are uniformly disposable in $ U_{\s} ^{ 2 } $ (resp. $ U_{ \kg }^{ 2 }  $), in the
      sense that
      \begin{align*}
        \lVert Q_{ j } f \rVert_{ U_{ \s }^{ 2 }  } \lesssim \lVert f \rVert  _{ U_{ \s }^{
        2}  } 
        \quad\text{and}\quad 
        \lVert Q_{ j } g \rVert_{ U_{ \kg }^{ 2 }  } \lesssim \lVert g \rVert  _{ U_{ \kg }^{
        2}  } 
      \end{align*} 
      for all $ f\in U_{ \s }^{ 2 }  $ and $ g\in U_{ \kg }^{ 2 }  $
      with implicit constant independent of $ j $.
      
    \item For any $ p > 2 $, we have the continuous embeddings
      \begin{align*}
        U_{\s}^{ 2 }
        \subset  V_{ \s }^{ 2 }  
        \subset  \left( U_{ \s }^{ p } \cap \dot X_{\s } ^{ 0,\frac{1}{2},\infty }\right)
        \quad\text{and}\quad 
        U_{\kg}^{ 2 }
        \subset  V_{ \kg }^{ 2 }  
        \subset  \left( U_{ \kg }^{ p } \cap \dot X_{\kg } ^{ 0,\frac{1}{2},\infty }\right).
      \end{align*} 
      
    \item Let $ (p,q) $ be a radially admissible Schrödinger pair (\ref{eq:ranges}) and 
      $ \phi\in U_{ \Delta }^{ 2 }  $, then
      \begin{align*}
        2^{ k\sigma_{ S }} \lVert P_{ k } \phi \rVert_{ L^{ p } L^{ q } } 
        &\lesssim \lVert \phi \rVert _{ U_{ \s }^{ 2 } } .
      \end{align*} 
      Moreover, if $ p > 2 $, 
      \begin{align*}
        2^{ k\sigma_{ S }  } \lVert P_{ k } \phi \rVert  _{ L^{ p } L^{ q }  } 
        &\lesssim \lVert \phi \rVert _{ V_{ \s }^{ 2 }  } .
      \end{align*} 

    \item Similarly, if $ (p,q) $ is radially admissible wave pair (\ref{eq:ranges}) and
      $ \psi\in U_{ \kg }^{ 2 }  $, then
      \begin{align*}
        2^{ k\sigma_{ W }} \lVert P_{ k } \psi \rVert_{ L^{ p } L^{ q } } 
        &\lesssim \lVert \psi \rVert _{ U_{ \kg }^{ 2 } } 
        \quad\text{and}\quad 
        2^{ k\sigma_{ W }  } \lVert P_{ k } \psi \rVert  _{ L^{ p } L^{ q }  } 
        \lesssim \lVert \psi \rVert _{ V_{ \kg }^{ 2 }  },
      \end{align*} 
      where the $ V^{ 2 }_{ \kg }  $ estimate only holds when $ p > 2 $.

    \item Let $ c \in \Z$ be a uniform frequency cutoff.
      If $ (p,q) $ is radially admissible Schrödinger pair (\ref{eq:ranges}) and 
      $ \psi\in U_{ \kg }^{ 2 }  $, then
      \begin{align*}
        2^{ k\sigma_{ S }} \lVert P_{ k } \psi \rVert_{ L^{ p } L^{ q } } 
        &\lesssim \lVert \psi \rVert _{ U_{ \kg }^{ 2 } } 
        \quad\text{and}\quad 
        2^{ k\sigma_{ S }  } \lVert P_{ k } \psi \rVert  _{ L^{ p } L^{ q }  } 
        \lesssim \lVert \psi \rVert _{ V_{ \kg }^{ 2 }  } 
      \end{align*}
      uniformly in $ k\leq c$, 
      where the $ V^{ 2 }_{ \kg }  $ estimate only holds when $ p > 2 $.
  \end{enumerate}
\end{prop}

\begin{proof}
  Proofs of properties (i) and (ii) can be found in \cite[Proposition 2.14, Corollary
  2.18]{HHK}.
  Property (iii) can be directly verified for radial atoms $ \phi = \sum_{ k=1 }^{ K }
  \1_{ [t_{ k-1 } ,t_{ k } ) } (t) e^{-it\Delta } \phi_{ k } $, 
  \begin{align*}
    2^{ kp\sigma_{ S } } \lVert \phi \rVert _{ L^{ p } L^{ q }  }^{ p } 
    &= 2^{ kp\sigma_{ S }  } \sum_{ k=1 }^{ K } \lVert e^{ -it\Delta } \phi_{ k }  \rVert _{
    L^{ p } L^{ q } }^{ p } 
    \lesssim \sum_{ k=1 }^{ K } \lVert \phi_{ k }  \rVert _{ L^{ 2 } _{ x }  }^{ p } = 1.
  \end{align*} 
  The equivalent statement for $ V_{ \s } ^{2  }  $ follows from the embedding $ V_{ \s } ^{
  2 }\subset U_{ \s }^{ p } $, $p > 2$.
\end{proof}

%Since the Strichartz estimates involve frequency localized functions, it is natural
%to consider frequency localized versions of $ U^{ 2 }  $ and $ V^{ 2 }  $.
%For $ k\in \Z $ and $ s\geq 0 $, we denote by $ U_{ \s,k }^{ 2,s }  $ (resp. $ V_{ \pm \kg, k   }^{
%2,s} $) the space of $ \phi\in U_{ \s }^{ 2,s }  $ such that $ \phi = \tilde P_{ k } \phi $
%(resp. $ \phi\in V_{ \pm\kg }^{ 2,s } $ such that $ \phi = \tilde P_{ k } \phi $ ).
%In particular, for any $ \phi\in U_{ \s,k }^{ 2,s }$ and $ \psi\in V_{ \s,k }^{ 2,s }  $,
%\begin{align*}
%  \lVert \phi \rVert _{ U_{ \s,k }^{ 2,s }  } = 2^{ ks } \lVert \phi \rVert_{ U_{ \s,k }^{ 2
%  }}
%  \quad\text{and}\quad 
%  \lVert \psi \rVert _{ V_{ \s,k }^{ 2,s } } = 2^{ ks } \lVert \psi \rVert_{ V_{ \s,k }^{2
%  }}.
%\end{align*} 
We are ready to state our estimates for inhomogeneous solutions.
%, for which we assume without loss of generality that $ s = 0 $.

\begin{lemma}\label{inhomogeneous}
  Let $ k\in\Z $, and let  $u_0, N_0  \in L^{ 2 } (\R^{ 3 };\C) $ 
  and $ f, g \in L_{ t, \emph{loc} }^{ 1
  }\left( \R; L_{ x }^{ 2 } (\R^{ 3 };\C)  \right) $ be functions localized at frequency $ 2^{ k }$. 
  Let
  \begin{align*}
    u(t) 
    &= e^{ -it\Delta} u_0 + \int_{ 0 }^{ t } e^{ -i(t-s)\Delta} f(s) \dif s,
    \quad
    N(t) 
    = e^{ it\langle D \rangle  } N_0 + \int_{ 0 }^{ t } e^{ i(t-s)\langle D \rangle  } g(s) \dif s.
  \end{align*} 
  Then, $ u = \tilde P_{ k } u $ and $ N = \tilde P_{ k } N $ are the unique solutions of
  \begin{align*}
    i\partial_{ t } u + \Delta u 
    &= f,
    \quad
    i\partial_{ t } N - \langle D \rangle  N 
    = g,
  \end{align*} 
  with
  $ u,N\in C\left( \R ; L^{ 2 } (\R^{ 3 };\C) \right)  $, and
  \begin{nalign}\label{eq:duality}
    \lVert u \rVert  _{ U_{ \s } ^{ 2 }  } 
    & \lesssim \lVert u_0 \rVert _{ L^{ 2 } (\R^{ 3 } ) } 
    + \sup_{ \lVert h \rVert_{V^2_{\s}}\leq 1} 
    \left\lvert \int_{ \R^{ 1+3 }  } \langle f, h \rangle \dif
    x\dif t \right\rvert,\\
    \lVert N \rVert  _{ U_{\kg } ^{ 2 }  } 
    & \lesssim \lVert N_0 \rVert _{ L^{ 2 } (\R^{ 3 } ) } 
    + \sup_{  \lVert h \rVert_{V^2_{\kg}}\leq 1  } \left\lvert \int_{ \R^{ 1+3 }  } \langle g, h \rangle \dif
    x\dif t \right\rvert,
  \end{nalign} 
  provided that $h \in L_{ t }^{ \infty } \left( \R;L^{ 2 } (\R^{ 3 }, \C)\right)$ is localized to frequency $\sim 2^k$ and the right-hand sides of (\ref{eq:duality}) are finite.
\end{lemma}
The lemma above is the standard $ U^{ 2 } $-duality theory, see \cite[Theorem
2.8]{HHK} and, for a frequency localized version, see 
\cite[Proposition 2.11]{HTT}.

%To construct our resolution space, we need to recombine the frequency localized spaces $ U_{
%k}^{ 2 }  $.
%For each $ 1\leq q < \infty  $, we define the norm
%\begin{align*}
%  \lVert \phi \rVert _{ \ell^{ q } U^{ 2 } } 
%  &\defeq\left( \sum_{ k } \lVert P_{ k } \phi \rVert_{ U_{ k }^{ 2 }  }^{ q }  \right) ^{
%  \frac{1}{q}  }.
%\end{align*} 
%The space $ \ell^{ q } U^{ 2 } $ is the completion of
%\textcolor{SkyBlue}{
%  I don't know what the best way of defining this is.
%}
%
%
%The proof of the proposition below is based on \cite[Proposition 4.3]{CH-wm}.
%
%\begin{prop}
%  We have the following continuous embeddings:
%  \begin{align*}
%    \ell^{ 1 } U^{ 2 } \subset  U^{ 2 } \subset \ell^{ 2 } U^{ 2 } \subset  V^{ 2 } .
%  \end{align*} 
%\end{prop}
%
%\begin{proof}
%  \textcolor{SkyBlue}{
%    Prove it.
%  }
%  
%\end{proof}

We define the resolution spaces $ Z^{ \s }_{ s } $ and $ Z_{ \kg }^{ s }  $ at regularity $
s \in\R $ by
\begin{align*}
  Z^{ s }_{ \s } 
  &\defeq \cl\left( \left\{ u\in C\left( \R; L_{ x }^{ 2 }  \right) \cap 
    U_{ \s } ^{ 2 } : \lVert P_{ \leq 0 } u \rVert  _{ U_{ \s }^{ 2 }  } 
    + \sum_{ k\geq 1 } 2^{ 2ks } \lVert P_{ k } u \rVert_{ U_{ \s }^{ 2 }} ^2< \infty
    \right\}  \right) \\
  Z^{ s }_{ \kg }  
  &\defeq \cl\left( \left\{ N\in C\left( \R; L_{ x }^{ 2 } \right) \cap 
    U_{ \kg } ^{ 2 } : \lVert P_{ \leq 0 } N \rVert  _{ U_{ \kg }^{ 2 } }
    +\sum_{ k\geq 1 } 2^{ 2ks } \lVert P_{ k } N \rVert_{ U_{ \kg }^{ 2
    }  }^2 < \infty \right\} \right) 
 \end{align*} 
with norms
\begin{align*}
    \lVert u \rVert_{ Z_{ \s }^{ s }   } 
  &\defeq \lVert P_{ < -10 } u \rVert_{ U_{ \s }^{ 2 } }
    + \left( \sum_{ k\geq -10 } 2^{ 2ks } \lVert P_{ k } u \rVert_{ U_{ \s }^{ 2 } }^{ 2 }
    \right) ^{ \frac{1}{2}  },\\
    \lVert u \rVert_{ Z_{ \kg }^{ s } } 
  &\defeq \lVert P_{ < -10 } u \rVert_{ U_{ \kg }^{ 2 } } 
    + \left( \sum_{ k\geq -10 } 2^{ 2ks } \lVert P_{ k } u \rVert_{ U_{ \kg }^{ 2 } }^{ 2 }
    \right) ^{ \frac{1}{2} },
\end{align*} 
where $\cl$ denotes the closure (completion) under the natural metric topology.

The choice of low frequency cutoff $ P_{ < -10 }  $ is somewhat unusual but necessary for
our multilinear analysis in Section \ref{section 4}. 
By Lemmas \ref{lemma:sch nonresonance} and \ref{lemma: kg nonresonance}, nonlinear
interactions at low enough frequencies are uniformly nonresonant and can be treated all at
once.
This ceases to be the case near frequency $ k = 1 $, which prevents us from using the more
typical $ P_{ < 0 }  $ cutoff.

In fact, the structures above are based on a more general Besov-type construction
for $ U^{ 2 } $ and $ V^{ 2 } $ spaces.
Let $ 1\leq p <\infty $.
The space $ \ell^{ p } U^{ 2 }$ is the completion of the set of $ u\in U^{ 2 }  $ such that
$ \sum_{ k\geq 1 } \lVert P_{ k } u \rVert_{ U^{ 2 } }^{ p } < \infty  $
with norm
\begin{align*}
  \lVert u \rVert  _{ \ell^{ p } U^{ 2 }  } 
  \defeq
  \left( \lVert P_{ < -10 } u \rVert_{ U^{ 2 }  }^{ p }
   + \sum_{ k\geq -10 } \lVert P_{ k } u \rVert_{ U^{ 2 }  } ^{ p } \right) ^{ \frac{1}{p}  }.
\end{align*} 
Similarly, we also define $ \ell^{ p } V^{ 2 }  $ by replacing the $ U^{ 2 }  $-based norms
above by $ V^{ 2 }  $ norms.
In particular, $ U^{ 2 }  $ and $ V^{ 2 }  $ satisfy a one-sided almost orthogonality
property, see \cite[Proposition 4.3]{CH-wm}.

\begin{prop}\label{prop:ellp}
  The following embeddings are continuous
  \begin{align*}
    %\ell^{ 1 } U^{ 2 } 
    %\subset  
    U^{ 2 } 
    \subset \ell^{ 2 } U^{ 2 } 
    \subset \ell^{ 2 } V^{ 2 } 
    \subset  V^{ 2 }.
  \end{align*} 
\end{prop}
\noindent
Below we have a direct consequence of Proposition \ref{prop:ellp} and our definition of $Z^s$.
\begin{coro}
    Let $s\in \R$ and let $\Phi = \s$ or $\kg$.
    Then,
    \begin{align}\label{eq:scattering}
        \lVert \langle D\rangle^s u \rVert_{V^2_\Phi}
        \lesssim
        \lVert \langle D\rangle^s u \rVert_{\ell^2 U^2_\Phi}
        \lesssim \lVert u \rVert_{Z^s_\Phi}
    \end{align}
    uniformly in $u\in Z^s_\Phi$.
\end{coro}

In particular, if $u\in Z^s_{\s}$ and $N\in Z^s_{\kg}$, by the definition of $V^2$ we conclude that the functions
\begin{align*}
    e^{it\Delta}u
    \quad\text{and}\quad
    e^{-it\langle D\rangle} N
\end{align*}
have $H^s_x$-limits as $t\to\infty$.
This observation will be a key step in the proof of scattering in Subsection \ref{mainproof}.

\section{Multilinear estimates}\label{section 4}
Our goal in this section is to find $ s,r\in\R $ such that the trilinear estimates
%%%% estimates without \overline N
\begin{align}
  \left\lvert \int_{\R^{1+3}} N u_{1}\overline u_{2}\dif x\dif t \right\rvert
  &\lesssim \lVert N \rVert _{ U_{ \kg }^{ 2,r } } \lVert u_{1} \rVert _{ U_{ \s}^{ 2,s } }
  \lVert u_{2} \rVert_{ V_{ \s }^{ 2,-s } },\label{eq:finalSchrodinger}\\
  \left\lvert \int_{\R^{1+3}} \langle D \rangle ^{-1}
    \left( u_{1}\overline u_{2}\right)\overline N \dif x \dif t\right\rvert
  &\lesssim \lVert N \rVert_{ V_{ \kg }^{ 2,-r } } \lVert u_{1} \rVert _{ U_{ \s }^{ 2,s } } \lVert u_{2} \rVert_{ U_{ \s }^{ 2,s } }\label{eq:finalKleinGordon},
\end{align}
holds for functions $N$, $u_1$ and $u_2$ localized at frequencies $\sim 2^k$, $2^{k_1}$ and $2^{k_2}$, respectively.

%%%%% estimates with \overline N
%\begin{align}
%  \left\lvert \int_{\R^{1+3}} N_{k}^{s}u_{1}\overline u_{k_{2}}\dif x\dif t \right\rvert
%  &\lesssim \lVert N_{k}^{s} \rVert _{ U_{ \pm \kg,k }^{ 2 } } \lVert u_{1} \rVert _{ U_{
%  \s, k_1}^{ 2 } }
%  \lVert u_{2} \rVert_{ V_{ \s, k_2 }^{ 2 } },\label{eq:finalSchrodinger}\\
%  \left\lvert \int_{\R^{1+3}} \langle D \rangle ^{-1}
%    \left( u_{1}\overline u_{2}\right)\bar N_{k} \dif x \dif t\right\rvert
%  &\lesssim \lVert N_{k} \rVert_{ V_{ \kg }^{ 2 } } \lVert u_{1} \rVert _{ U_{ \s,k_1
%  }^{ 2 } } \lVert u_{2} \rVert_{ U_{ \s }^{ 2 } }\label{eq:finalKleinGordon}
%\end{align}
%for frequency localized functions $N = \tilde P_{k}N$, $u_{i} = \tilde P_{k_{i}}u_{i}$,
%$s\in\left\{\pm\right\}$, and $N^{+} \defeq N$, $N^{-} \defeq \overline N$.

\subsection{Bilinear restriction estimates}
We record here the bilinear estimates that we will use to control certain resonant interactions. First, we briefly summarize the necessary curvature and transversality conditions from \cite{Ca} that need to be checked. We refer the interested reader to Corollary 1.6 and Lemma 2.1 (and the related discussions) in \cite{Ca} for more details.

Let $j=1,2$. Define phases  $\Phi_j$ for the flows and Fourier supports $\Lambda_j\subset\R^n$ for the two waves, respectively. For our applications, the $\Phi_j$ will be either Schrödinger or Klein-Gordon phases and the $\Lambda_j$ will be either balls at the origin or slight perturbations of dyadic annuli.

Define the quantities:
\begin{align*}
    \mathcal{V}_{max}&\defeq\sup_{\substack{\xi\in\Lambda_1\\\eta\in\Lambda_2}}\abs{\nabla\Phi_1(\xi)-\nabla\Phi_2(\eta)},
    \qquad\mathcal{H}_j\defeq\norm{\nabla^2\Phi_j}_{L^\infty(\Lambda_j)},
\end{align*}
and let $d_0>0$ be a constant which will be quantified more precisely below. The derivative gains and losses in the bilinear estimate are expressed in terms of these three parameters. For our applications, $\mathcal{V}_{max}$ will be the high frequency, $d_0$ will be the low frequency, and $\mathcal{H}_j\sim 1$.

We are now ready to state the key curvature and transversality hypotheses for the bilinear restriction estimate to hold:
\begin{enumerate}
    \item [\textbf{(A1')}] For every $\{j,k\}=\{1,2\}$, $\xi\in\Lambda_j$, $\eta\in\Lambda_k$, and $v\in\R^n$ with $v\cdot(\nabla\Phi_j(\xi)-\nabla\Phi_k(\eta))=0$, we have:
    \begin{align*}
        \abs{[\nabla^2\Phi_j(\xi)v] \wedge (\nabla\Phi_j(\xi)-\nabla\Phi_k(\eta))}\gtrsim\mathcal{H}_j\mathcal{V}_{max}\abs{v},
    \end{align*}
    \item [\textbf{(A2)}] For every $j=1,2$, we have $\Phi_j\in C^{5n}(\Lambda_j)$ and:
    \begin{align*}
        \min_{3\leq m\leq 5n}\left(\frac{\mathcal{H}_j}{\norm{\nabla^m\Phi_j}_{L^\infty(\Lambda_j)}}\right)^{\frac{1}{m-2}}\geq d_0, \qquad\frac{\mathcal{V}_{max}}{\mathcal{H}_j}\geq d_0.
    \end{align*}
\end{enumerate}
We remark that Condition \textbf{(A1')} is a local version of the more general global curvature and transversality hypothesis Condition \textbf{(A1)} from \cite{Ca}. By \cite[Lemma 2.1]{Ca}, the local condition implies the global condition in our setting so we just state the local version for simplicity.

The bilinear restriction estimate then reads:
\begin{thm}[Corollary 1.6 in \cite{Ca}]\label{thm:bilinearrange}
Let $n\geq 2$, $1\leq q,r\leq 2$, and $\frac{1}{q}+\frac{n+1}{2r}<\frac{n+1}{2}$. Define all the above objects and quantities and let $u\in U^2_{\Phi_1}$ and $v\in U^2_{\Phi_2}$ with $\mathcal{H}_2\leq\mathcal{H}_1$ and
\begin{align*}
    \on{supp}(\widehat{u})+B(0,d_0)\subset\Lambda_1, && \on{supp}(\widehat{v})+B(0,d_0)\subset\Lambda_2, && \min\{\on{diam}(\on{supp}(\widehat{u})), \on{diam}(\on{supp}(\widehat{v}))\}\leq d_0.
\end{align*}
Then, the following bilinear estimate holds:
\begin{align*}
    \norm{uv}_{L^q_tL^r_x(\R^{n+1})}\lesssim d_0^{n+1-\frac{n+1}{r}-\frac{2}{q}}\mathcal{V}_{max}^{\frac{1}{r}-1}\mathcal{H}_1^{1-\frac{1}{q}-\frac{1}{r}}\left(\frac{\mathcal{H}_1}{\mathcal{H}_2}\right)^{\frac{1}{q}-\frac{1}{2}}\norm{u}_{U^2_{\Phi_1}}\norm{v}_{U^2_{\Phi_2}}
\end{align*}
\end{thm}

%Given $\mathfrak{h}=(h,a)\in\R\times\R^n$ and $\{j,k\}=\{1,2\}$, define the surface:
%\begin{align*}
%    \Sigma_j(\mathfrak{h})\defeq\{\xi\in\Lambda_j\cap(h-\Lambda_k)\mid \Phi_j(\xi)+\Phi_k(h-\xi)=a\}
%\end{align*}

%\begin{lemma}[Lemma 2.1 in \cite{Ca}]
%Let $\{j,k\}=\{1,2\}$. Suppose that for all $\mathfrak{h}\in\R^{n+1}$ we have
%\begin{align*}
%    \on{ConvexHull}[\Sigma_j(\mathfrak{h})]\subset\Lambda_j\cap (h-\Lambda_k).
%\end{align*}
%If in addition we have the small variation bound
%\begin{align*}
%    \sup_{\xi,\xi'\in\Lambda_1}\abs{\nabla\Phi_1(\xi)-\nabla\Phi_1(\xi')}+\sup_{\eta,\eta'\in\Lambda_1}\abs{\nabla\Phi_2(\eta)-\nabla\Phi_2(\eta')}\lesssim\mathcal{V}_{max},
%\end{align*}
%and for all $\xi,\xi'\in\Sigma_j(\mathfrak{h})$
%\begin{align*}
%    \abs{\nabla\Phi_j(\xi)-\nabla\Phi_j(\xi')-\nabla^2\Phi_j(\xi)(\xi-\xi')}\lesssim\mathcal{H}_j\abs{\xi-\xi'},
%\end{align*}
%then the local assumption \textbf{(A1')} implies the global condition \textbf{(A1)}.
%\end{lemma}

\begin{prop}\label{thm:candybilinear}
Let $N_k$, $u_{k_1}$, and $u_{k_2}$ be defined as above. Then the following estimates hold:
\begin{enumerate}[label=\upshape(\roman*)]
    \item If $-3\leq k\leq 7$ and $k_1\prec k$ then:
    \begin{align}\label{eq:candybilinear 1}
        \norm{N_k u_{k_1}}_{L^{\frac{8}{5}}_t L^{\frac{3}{2}}_x}\lesssim 2^{\frac{k_1}{12}}\norm{N_k}_{U^2_{\kg}}\norm{u_{k_1}}_{U^2_{\s}}
    \end{align}

    \item If $k_1>-3$, $k\leq 10$, and $k\prec k_1$ then:
    \begin{align}\label{eq:candybilinear 2}
        \norm{N_k u_{k_1}}_{L^{\frac{8}{5}}_t L^{\frac{3}{2}}_x}\lesssim 2^{\frac{k}{12}-\frac{k_1}{3}}\norm{N_k}_{U^2_{\kg}}\norm{u_{k_1}}_{U^2_{\s}}
    \end{align}

    \item If $-3\leq k_1\leq 10$ and $k_2\prec k_1$ then:
    \begin{align}\label{eq:candybilinear 3}
        \norm{u_{k_1} u_{k_2}}_{L^{\frac{8}{5}}_t L^{\frac{3}{2}}_x}\lesssim 2^{\frac{k_2}{12}}\norm{u_{k_1}}_{U^2_{\s}}\norm{u_{k_2}}_{U^2_{\s}}
    \end{align}
\end{enumerate}
\end{prop}

\begin{proof}
These estimates follow immediately from \cite[Corollary 1.6]{Ca} upon verifying that the necessary curvature and transversality conditions hold. We will just sketch these computations for the rest of this proof as they are generally very straightforward.
\begin{enumerate}[label=\upshape(\roman*)]
    \item Let $\Phi_1(\xi)=\inner{\xi}$ and $\Phi_2(\xi)=\abs{\xi}^2$. Fix $k$ and $k_1$, let $\Lambda_1=\{\abs{\xi}\sim 2^k\}+B(0,2^{k_1})$, and let $\Lambda_2=\{\abs{\xi}\lesssim 2^{k_1}\}$.

    We set $d_0=2^{k_1}$ and we note that the support condition clearly holds and we have, uniformly in $k$ and $k_1$, the estimates for $\mathcal{V}_{max}$ and $\mathcal{H}_j$:
    \begin{align*}
        \mathcal{V}_{max}&=\sup_{\substack{\xi\in\Lambda_1\\\eta\in\Lambda_2}}\abs{\frac{\xi}{\inner{\xi}}-2\eta}\sim 1\\
        \mathcal{H}_j&\sim 1, \quad j=1,2
    \end{align*}
    which leads to the corresponding derivatives on the right hand side of the estimate.

    Condition \textbf{(A1')} follows from the fact that the Hessian for both phase functions in these regions is either equal to or is approximately equal to the identity. Condition \textbf{(A2)} follows immediately from the fact that the higher order derivatives for these phases are either zero or essentially zero.

    \item Use the same phases as in part (i) but let $\Lambda_1=\{\abs{\xi}\lesssim 2^k\}$ and let $\Lambda_2=\{\abs{\xi}\sim 2^{k_1}\}+B(0,2^k)$. Set $d_0=2^k$ and note that $\mathcal{V}_{max}\sim 2^{k_1}$. The verification of Assumption 1.1 follows a very similar computation as in the first case.

    \item Let $\Phi_1(\xi)=\Phi_2(\xi)=\abs{\xi}^2$, let $\Lambda_1=\{\abs{\xi}\sim 2^{k_1}\}+B(0,2^{k_2})$, and let $\Lambda_2=\{\abs{\xi}\lesssim 2^{k_2}\}$. Set $d_0=2^{k_2}$ and note that $\mathcal{V}_{max}\sim 2^{k_1}$. The verification of Assumption 1.1 again follows a very similar computation as in the first case.
\end{enumerate}
\end{proof}

\subsection{Trilinear estimates} 

The rest of the paper is dedicated to the proofs of (\ref{eq:finalSchrodinger}) and
(\ref{eq:finalKleinGordon}).
For each nonlinearity, we split our proof into cases depending on the relationship between
$k$, $k_1$, and $ k_2 $. 
We organize the different cases in Figures (\ref{fig:schrodinger}) and
(\ref{fig:kleingordon}) by color, according to which approach will be used.
Effectively, we consider three distinct regimes:
\begin{itemize}
  \item Nonresonant regimes (cf. Lemmas \ref{lemma:sch nonresonance} and \ref{lemma: kg nonresonance}) are colored blue and comprise cases (I) and (III.1) in Figures
    \ref{fig:schrodinger} and \ref{fig:kleingordon}.
    This case corresponds to bilinear  $ L_{ t,x }^{ 2 }  $-estimates, and we use the modulation structure $ \dot X^{ 0,\frac{1}{2} }  $.

  \item The gold regions correspond to regimes where we use the extended radial Strichartz range.
    This corresponds to cases (II) and (IV.1), 
    in Figure \ref{fig:schrodinger}, and case (II) in Figure \ref{fig:kleingordon}. 

  \item Transversal cases are colored yellow. 
    In this regime the interacting waves in the nonlinearity are transversal.
    Geometrically, this means that the waves interact for less time, and we have access to
    stronger bilinear restriction estimates.
    This is especially useful in some low frequency regimes, 
    as it avoids derivative losses in both the Schrödinger and Klein-Gordon components.
    This case corresponds to regions (III.2) and (IV.2)
    in Figure \ref{fig:schrodinger}, and case (III.2) in Figure \ref{fig:kleingordon}. 
\end{itemize}

\begin{figure}[ht]
    \centering
    \begin{subfigure}{.5\textwidth}
        \centering
          \tiny
          \def\svgwidth{.97\textwidth}
          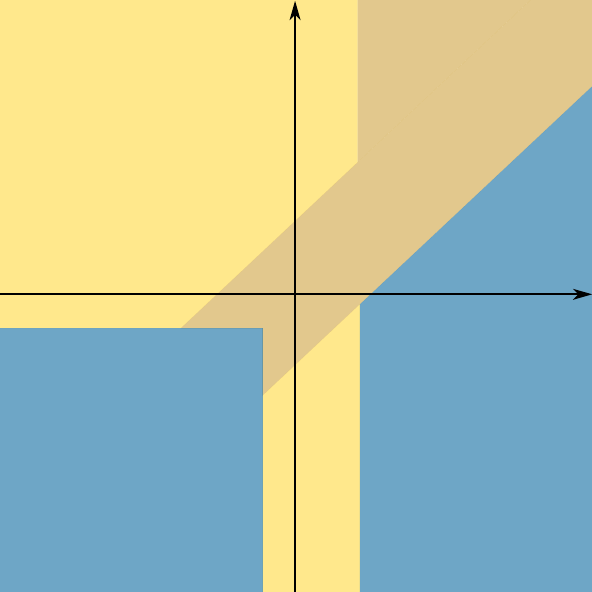
          \vspace{2pt}
          \caption{}
         \label{fig:schrodinger}
    \end{subfigure}%
    \begin{subfigure}{.5\textwidth}
        \centering
          \tiny
          \def\svgwidth{.97\textwidth}
          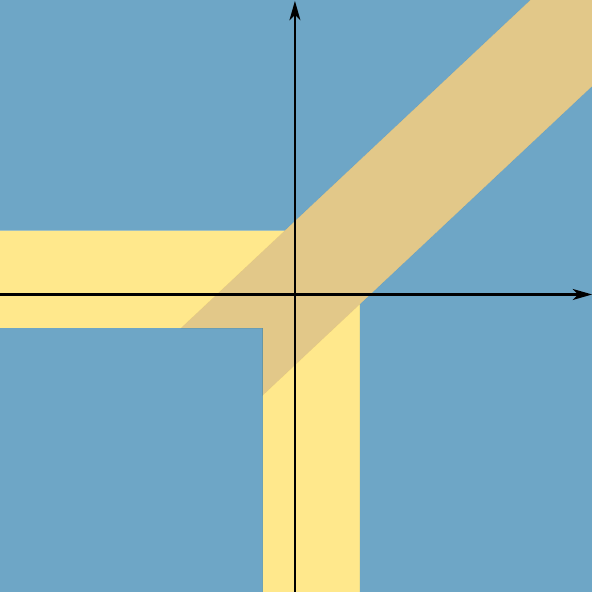
          \vspace{2pt}
          \caption{}
         \label{fig:kleingordon}
    \end{subfigure}%
    \caption{A visual breakdown of the cases in Propositions \ref{trilinearS}} and
    \ref{trilinearKG}, respectively.
\end{figure}

In what follows we adopt the $ + $ sign convention on the Klein-Gordon flow, simply writing
$ U_{ \kg }^{ 2 }$, which does not alter the trilinear estimates.
Similarly, we do not differentiate between $ N $ and $ \overline N $ in the nonlinear
analysis and always write $ N $, as the difference will be irrelevant in our estimates.

\subsection{Trilinear estimates -- Schrödinger nonlinearity}

\begin{prop}\label{trilinearS}
  Let $ \epsilon> 0 $ be sufficiently small, and let $N$, $u_1$ and $u_2$ be functions localized at frequency $\sim 2^k$, $2^{k_1}$ and $2^{k_2}$.  
  Then,
  \begin{align*}
    \left\lvert \int_{\R^{1+3}} Nu_{1}u_{2}\dif x\dif t \right\rvert
    \lesssim_{ \epsilon} \min \left\{ 1,2^{ \left( -\frac{1}{2} +\frac{\epsilon}{2} \right) k }
    \right\}  
    \lVert N \rVert_{U_{ \kg }^{ 2 } } \lVert u_{1} \rVert _{ U_{ \s }^{ 2 } } \lVert
    u_{2} \rVert _{V_{ \s }^{ 2 } }.
  \end{align*}
  When $ k < -10$ the same estimate holds when $u = \tilde P_{< -10} u$; similarly for $ k_1 $ and $ k_2 $.
\end{prop}

Before proceeding with the proof, we discuss a technical issue related to our choices of
Lebesgue space exponents.
Throughout the proof of Proposition \ref{trilinearS} we will often use Hölder's inequality
with mixed $ L_{ t }^{ p }L_{ x }^{ q }  $-spaces
\begin{align}\label{eq:toygeneral}
  \left\lvert \int_{ \R^{ 1+3 } } N u_1 u_2 \dif x\dif t  \right\rvert 
  &\lesssim \lVert N \rVert  _{ L_{ t } ^{ ^{ p }  }L_{ x }^{ q }   }
  \lVert u_1 \rVert  _{ L_{ t }^{ p_1 } L_{ x }^{ q_1 }  } \lVert u_2 \rVert  _{ L_{ t }^{
  p_2 } L_{ x }^{ q_2 } }
\end{align} 
for various combinations of admissible exponents $ (p,q)$, $(p_1,q_1) $, and $(p_2,q_2)$.
By Proposition \ref{properties}, these spaces can then be embedded into $ U^{ 2 }  $ and $
V^{ 2 } $ at the cost of dyadic powers depending on the frequencies $ k,k_1,k_2 $.  
In many cases, the optimal choice of exponents that minimizes the coefficient on the
right-hand side of (\ref{eq:toygeneral}) is based on the heuristic endpoint Strichartz combination
\begin{align}\label{eq:toy}
  \left\lvert \int_{ \R^{ 1+3 } } N u_1 u_2 \dif x\dif t  \right\rvert 
  &\lesssim \lVert N \rVert  _{ L_{ t } ^{ ^{ \infty }  }L_{ x }^{ 2 }   }
  \lVert u_1 \rVert  _{ L_{ t }^{ 2 } L_{ x }^{ \frac{10}{3} }  } \lVert u_2 \rVert  _{ L_{ t }^{
  2} L_{ x }^{ 5 } }
\end{align} 
%\begin{align*}
%  \left( p,q \right) = \left( \infty,2 \right),
%  \quad
%  \left( p_1,q_1 \right) = \left( 2,\frac{10}{3}  \right),
%  \quad\text{and}\quad 
%  \left( p_2,q_2 \right) = (2,5)
%\end{align*} 
which, ideally, would embed into $ U^{ 2 }  $ and $ V^{ 2 }  $ with a gain
\begin{align*}
  \left\lvert \int_{ \R^{ 1+3 } } N u_1 u_2 \dif x\dif t  \right\rvert 
  &\lesssim 2^{ -\frac{2k_1}{5} -\frac{k_2}{10}  } 
    \lVert N \rVert_{U_{ \kg }^{ 2 } } \lVert u_{1} \rVert _{ U_{ \s,k_1 }^{ 2 } } \lVert
    u_{2} \rVert _{V_{ \s }^{ 2 } }.
\end{align*} 
When $ k_1\sim k_2\geq 2 $, this gain corresponds to the estimate for $ \epsilon = 0 $ in
Proposition \ref{trilinearS}, which is the missing regularity endpoint in Theorem \ref{mainthm}.
However, this choice fails for two reasons:
the endpoint Strichartz estimate for $ (2,10 /3) $ fails even at the level of free
solutions, see Proposition \ref{linearstructure}. Moreover, choosing $ p_2 = 2 $ prevents us from controlling $ u_2 $ in $ V^{ 2 }$, since the embedding $V^2_{\Phi}\subset L_t^p L_x^q$ fails at $p = 2$ for both the Schrödinger and Klein-Gordon flows $\Phi = \Delta$ and $\Phi =\kg$.
We overcome these shortcomings by choosing Strichartz exponents close to the endpoint, but
still inside the admissible range, at the cost of incurring $ 2 ^{ \epsilon k}$-losses in
Proposition \ref{trilinearS} and the main theorem (\ref{mainthm}). See (\ref{eq:endpointperturb}) for an example.

\begin{proof}[Proof of Proposition \ref{trilinearS}]
  We split the proof into different cases depending on the relationship between $k,k_1,k_2$.
  Fix an absolute frequency cutoff $c$, which we set as $c = -3$.
  
  %%%%%%%%%%%%%%%%%%%%%%% CASE 1 %%%%%%%%%%%%%%%%%%%%%%%
  \textbf{Case I}: $\max\left\{k,k_1,k_2\right\}< c$.

  By Lemma (\ref{lemma:sch nonresonance}), we have $\lvert \Phi_{\pm}(\xi,\eta) \rvert\sim 1$, which means that all interactions produce modulations at scale $\gtrsim 1$. We will use this same fact repeatedly throughout the argument. So, we decompose
  \[
    I \defeq \int Nu_1 u_2\dif x\dif t = I_0 + I_1 + I_2,
  \] 
  where 
  \begin{align*}
    &I_0 = \sum_{j\gtrsim 1} \int Q_{ j} N Q_{\leq j}u_1 Q_{\leq j}u_{ 2}\dif x \dif t, \\
    &I_1 = \sum_{j\gtrsim 1} \int Q_{< j} N Q_{ j}u_1 Q_{< j}u_{ 2}\dif x \dif t, \\
    &I_2 = \sum_{j\gtrsim 1} \int Q_{< j} N Q_{< j}u_1 Q_{ j}u_{2}\dif x \dif t.
  \end{align*}
  Note that we are slightly abusing notation here as we should add a superscript to the modulation cutoff $Q_{j}$ to indicate
  whether it refers to the modulation with respect to the paraboloid $Q_{j}^{\s}$ or the
  Klein-Gordon characteristic surface $Q_{j}^{\kg}$. For simplicity, we omit the notation by
  assuming that $Q_{j}N = Q_{j}^{\kg}N$ and $Q_{j}u_{i} = Q_{j}^{\s}u_{i}$, $i
  =1,2$.

  Then, by Bernstein's inequality,
  \begin{align}
    \lvert I_0 \rvert 
    &\lesssim \sum_{j\geq 1} \lVert Q_{j}N \rVert _{L_{t,x}^{2}} \lVert Q_{\leq j}u_1 \rVert
    _{L_{t,x}^{4}} \lVert Q_{\leq j} u_2 \rVert _{L_{t,x}^{4}}\nonumber\\
    &\lesssim 2^{ \frac{3}{20} \left( k_1 + k_2 \right)  }\sum_{j\geq 1}
    2^{-\frac{j}{2}+\frac j{10}}\lVert Q_{j}N \rVert _{\dot X^{0,\frac{1}{2}}}
    \lVert Q_{\leq j}u_1 \rVert _{L_{t,x}^{\frac{10}{3}}} 
    \lVert Q_{\leq j} u_2 \rVert _{L_{t,x}^{\frac{10}{3}}}\label{eq:ll}\\
    &\lesssim 2^{ \frac{3}{20} \left( k_1 + k_2 \right)  } \lVert N \rVert _{ U_{ \kg }^{
    2 } } \lVert u_1 \rVert _{ U_{ \s}^{ 2 } }\lVert u_{2} \rVert_{ V_{ \s }^{ 2 }
    }\nonumber
  \end{align}
  Similarly, by Hölder and Bernstein's inequalities,
  \[
    \lvert I_1 \rvert \lesssim \sum_{j\geq 1}2^{-\frac{j}{2}} \lVert Q_{j}u_1
    \rVert_{L_{t,x}^{2}}
    \lVert Q_{< j}N \rVert _{L_{t,x}^{4}}\lVert Q_{< j}u_2 \rVert _{L_{t,x}^{4}}
    \lesssim 2^{ \frac{k_2}{20}  } \lVert N \rVert _{U_{ \kg }^{ 2 } }
    \lVert u_1 \rVert _{ U_{\s }^{2  } }\lVert u_2 \rVert
    _{ V_{ \s }^{ 2 } }.
  \] 
  The bound for $I_2$ is analogous.
  
  In all the remaining cases we will implicitly assume that $ \med (k,k_1,k_2) \sim \max
  (k,k_1,k_2) $ and that $ \max\left\{ k, k_1, k_2 \right\} \gtrsim c $.

  %%%%%%%%%%%%%%%%%%%%%% CASE 2 %%%%%%%%%%%%%%%%%%%%%%%
  \textbf{Case II}: $\lvert k-k_1 \rvert \leq 10$.
    We split this regime into two further subcases based on the size of $k_2$ (not
    pictured in Figure \ref{fig:schrodinger}) in order to take advantage of the nonresonance
    when $k_2$ is much smaller than $k$ and $k_1$.

    \emph{Case II.1}: If $k\sim k_1\sim k_2$ then we use the increased radial Strichartz range to estimate
%    \begin{align}
%        \abs{I(\textbf{k})}\lesssim \norm{N}_{L^{\infty-}_t L^{2+}_x}\norm{u_1}_{L^{2}_t L^{\frac{10}{3}+}_x}\norm{u_2}_{L^{2+}_t L^{5-}_x}\lesssim 2^{-\frac{k}{2}+}\norm{N}_{U^2_{\kg,k}}\norm{u_1}_{U^2_{\s,k_1}}\norm{u_2}_{V^2_{\s,k_2}}
%    \end{align}
    % New estimate to be
    \begin{nalign}\label{eq:endpointperturb}
        \abs{I}
        &\lesssim \norm{N}_{L^{p}_t L^{q}_x}\norm{u_1}_{L^{p_1}_t
          L^{q_1}_x}\norm{u_2}_{L^{p_2}_t L^{q_2}_x}\\
        &\lesssim 2^{ \epsilon k - \frac{ (8-3\epsilon) }{ 2(10 + 3\epsilon) }k_1
          -\frac{ (2+16\epsilon + 3\epsilon^2 )}{ 2(10+3\epsilon)  }k_2 } 
          \norm{N}_{U^2_{\kg}}\norm{u_1}_{U^2_{\s}}\norm{u_2}_{V^2_{\s}}\\
        &\lesssim
          2^{\left( -\frac{1}{2} +\frac{\epsilon}{2} \right)k}
          \norm{N}_{U^2_{\kg}}\norm{u_1}_{U^2_{\s}}\norm{u_2}_{V^2_{\s}}\\
    \end{nalign}
    where the choice of exponents below is a small perturbation of the heuristic
    estimate (\ref{eq:toy}) which moves the endpoint spaces into the admissible region (see
    Figure \ref{fig:Strichartz}):
    \begin{align*}
      \left( p,q \right) = \left( \frac{2}{\epsilon}, \frac{2}{1-\epsilon}  \right),
      \,
      \left( p_1,q_1 \right) = \left(2,\frac{10}{3} +\epsilon\right) ,
      \,\text{and}\,
      \left( p_2,q_2 \right) = \left( \frac{2}{1-\epsilon}, \left( \frac{1}{2} +
      \frac{\epsilon}{2} -\frac{ 3 }{ 10+3\epsilon } \right)^{-1} \right).
    \end{align*} 

    \emph{Case II.2}: If $k\sim k_1\succ k_2$ we use the nonresonance to treat this estimate similarly to Case I. In particular, by Lemma \ref{lemma:sch nonresonance} we have $\abs{\Phi_{\pm}(\xi,\eta)}\gtrsim 2^{2k}$ and so we decompose:
    \begin{align*}
        I \defeq \int Nu_1 u_2\dif x\dif t = I_0 + I_1 + I_2,
    \end{align*}
    where
    \begin{align*}
      &I_0 = \sum_{j\gtrsim 2k} \int Q_{ j} N Q_{\leq j}u_1 Q_{\leq j}u_{ 2}\dif x \dif t, \\
      &I_1 = \sum_{j\gtrsim 2k} \int Q_{< j} N Q_{ j}u_1 Q_{< j}u_{ 2}\dif x \dif t, \\
      &I_2 = \sum_{j\gtrsim 2k} \int Q_{< j} N Q_{< j}u_1 Q_{ j}u_{2}\dif x \dif t.
    \end{align*}
    For $I_0$ we can use the same estimate as Case I to get:
    \begin{align*}
        \abs{I_0}&\lesssim \sum_{j\gtrsim 2k} \norm{Q_jN}_{L^2_{t,x}} \norm{Q_{\leq j}u_1}_{L^4_{t,x}} \norm{Q_{\leq j}u_2}_{L^4_{t,x}}\\
        &\lesssim 2^{\frac{3}{20}(k_1+k_2)} \sum_{j\gtrsim 2k} 2^{-\frac{j}{2}+\frac{j}{10}}\norm{Q_j N}_{\dot{X}^{0,\frac{1}{2}}} \norm{Q_{\leq j}u_1}_{L^{\frac{10}{3}}_{t,x}} \norm{Q_{\leq j}u_2}_{L^{\frac{10}{3}}_{t,x}}\\
        &\lesssim 2^{\frac{3}{20}(k_1+k_2)}2^{-k+\frac{k}{5}} \norm{N}_{U^2_{\kg}} \norm{u_1}_{U^2_{\s}} \norm{u_2}_{V^2_{\s}}\\
        &\leq 2^{-\frac{k}{2}} \norm{N}_{U^2_{\kg}} \norm{u_1}_{U^2_{\s}} \norm{u_2}_{V^2_{\s}}
    \end{align*}
    For $I_1$, we perturb around the sharp endpoint combination $ (p,q) = (\infty,3) $,
    $(p_2,q_2) = (2,6)$ obtaining
    \begin{nalign}\label{eq:hhL1L2}
        \abs{I_1}
        &\lesssim \sum_{j\gtrsim 2k} \norm{Q_{\leq j}N}_{L^{\frac{ 2(2+\delta) }{ \delta } }_t L^3_x}
          \norm{Q_{j}u_1}_{L^2_{t,x}} \norm{Q_{\leq j}u_2}_{L^{2 + \delta}_t L^6_x}\\
        &\lesssim \sum_{j\gtrsim 2k} 2^{-\frac{j}{2}} \norm{Q_{\leq j}N}_{L^{ \frac{
          2(2+\delta) }{ \delta } }_t L^3_x} \norm{Q_{j}u_1}_{\dot{X}^{0,\frac{1}{2}}}
          \norm{Q_{\leq j}u_2}_{L^{2+\delta}_t
          L^6_x}\\
        &\lesssim 2^{-k + \frac{1}{2+\delta}k +\frac{ \delta }{ 2+\delta }k_2 }
          \norm{N}_{U^2_{\kg}} \norm{u_1}_{U^2_{\s}} \norm{u_2}_{V^2_{\s}}
    \end{nalign}
    where the coefficient satisfies
    \begin{align*}
      2^{-k + \frac{1}{2+\delta}k +\frac{ \delta }{ 2+\delta }k_2 }
      \lesssim 2^{ -\frac{1}{2+\delta} k } 
      \lesssim 2^{ \left( -\frac{1}{2} +\frac{\epsilon}{2}  \right) k } 
    \end{align*} 
    provided that $0 < \delta \leq \frac{ 2\epsilon }{ 1-\epsilon }  $ and $ k\succ k_2 $.
    For $I_2$ we can repeat the same estimate with $ \delta = 0 $.
    %Thus, in both cases, we get the estimate:
    %\begin{align}
    %  \lvert I ( \textbf{k} )\rvert 
    %  \lesssim \lVert u_2 \rVert _{ L^{ \infty } L^{ 2 } } 
    %  \lVert N \rVert _{ L^{2} L^{p(-\epsilon) }}
    %  \lVert u_1 \rVert _{ L^{ 2 } L^{ p(\epsilon) }  }
    %   \lesssim 2^{-\frac{k}{2}+} \lVert N \rVert
    %   _{ U_{ \kg }^{ 2 } } \lVert u_1 \rVert
    %   _{ U_{ \s }^{ 2 } } \lVert u_2 \rVert _{ V_{ \s }^{ 2 } },
    %\end{align}
    %which is acceptable.

  %%%%%%%%%%%%%%%%%%%%%%% CASE 3 %%%%%%%%%%%%%%%%%%%%%%%
  \textbf{Case III}: $ k \succ k_1$.
  \emph{Case III.1}: If $ k>7 $, due to  Lemma \ref{lemma:sch nonresonance} $I$ only has nontrivial
  contributions when $j\gtrsim 2k$, and we directly estimate
  \begin{align*}
    \lvert I_{2} \rvert 
    &\lesssim \sum_{j\gtrsim 2k}  \lVert Q_{\leq j}N \rVert _{L^{\infty}_{t}L^{3}_{x}}
      \lVert Q_{\leq j} u_{1} \rVert_{L^{2}_{t}L_{x}^{6}}\lVert Q_{ j } u_{2}
      \rVert_{L^{2}_{t,x}} \\
     &\lesssim \sum_{j\gtrsim 2k}2^{-\frac{j}{2}} \lVert Q_{\leq j}N \rVert _{L^{\infty}
       L^{3}} \lVert Q_{\leq j}u_{1} \rVert _{L^{2}L^{6}}\lVert Q_{ j } u_{2} \rVert_{\dot
       X^{0,\frac{1}{2}}}\\
     &\lesssim \sum_{j\gtrsim 2k}2^{ \frac{1}{2} (k-j)}        \lVert Q_{\leq j}N \rVert
       _{L^{\infty}L^{2}} \lVert Q_{\leq j}u_{1} \rVert _{L^{2}L^{6}}\lVert Q_{ j } u_{2}
       \rVert_{\dot X^{0,\frac{1}{2}}}\\
    &\lesssim 2^{ -\frac{k}{2} } \lVert N \rVert _{ U_{ \kg }^{ 2 } }
    \lVert u_{1} \rVert _{ U_{ \Delta }^{ 2 } } 
    \lVert u_2 \rVert _{ V_{ \Delta }^{ 2 } },
  \end{align*}
  and $ I_0 $ can be estimated using the same spaces by switching the roles of $ N  $ and 
  $ u_1$.
  Estimating $ I_{ 1 } $ as in (\ref{eq:hhL1L2}), we obtain
  \begin{align*}
    \lvert I_{ 1 } \rvert  
    &\lesssim 2^{\left( -\frac{1}{2} + \frac{\epsilon}{2} \right)k  }
          \norm{N}_{U^2_{\kg}} \norm{u_1}_{U^2_{\s}} \norm{u_2}_{V^2_{\s}}.
  \end{align*} 
  \emph{Case III.2}: If $-3\leq k\leq 7$, we use the bilinear restriction estimate
  (\ref{eq:candybilinear 1}),
  \begin{align*}
    \lvert I \rvert  
    & \lesssim  \lVert N u_1 \rVert  _{ L^{ \frac{8}{5} } L^{ \frac{3}{2} } } 
      \lVert  u_2 \rVert _{ L^{ \frac{8}{3} } L^{ 3 } } 
      \lesssim 2^{\frac{k_1}{12}-\frac{k_2}{4}  } \lVert N \rVert  _{ U_{\kg }^{ 2 }  }
      \lVert u_1 \rVert  _{ U_{ \s } ^{ 2 }  } 
      \lVert u_2 \rVert _{ V_{ \s }^{ 2 } },
  \end{align*} 
  which is acceptable since $ k_2 \sim k = O(1)$ and $ k \succ k_1 $.

  %%%%%%%%%%%%%%%%%%%%%%% CASE 4 %%%%%%%%%%%%%%%%%%%%%%%
  \textbf{Case IV}: $ k_1\succ k $.
   %Again, we consider the cases $ k\succ 1 $ and $ k\lesssim 1 $ separately.
   \emph{Case IV.1}: $ k \geq 10 $.  We can proceed as in (\ref{eq:endpointperturb}) and
   use the increased radial
   Strichartz range: 
  %\begin{align*} 
  %  \abs{I(\textbf{k})}\lesssim \norm{N}_{L^{\infty-}_t
  %  L^{2+}_x}\norm{u_1}_{L^{2}_t L^{\frac{10}{3}+}_x}\norm{u_2}_{L^{2+}_t L^{5-}_x}\lesssim
  %  2^{-\frac{k_2}{2}+}\norm{N}_{U^2_{\kg,k}}\norm{u_1}_{U^2_{\s,k_1}}\norm{u_2}_{V^2_{\s,k_2}}.
  %\end{align*}
  \begin{align*}
      \abs{I}
      \lesssim \norm{N}_{L^{p}_t L^{q}_x}\norm{u_1}_{L^{p_1}_t
        L^{q_1}_x}\norm{u_2}_{L^{p_2}_t L^{q_2}_x}
      \lesssim
        2^{\left( -\frac{1}{2} +\frac{\epsilon}{2}   \right)k}
        \norm{N}_{U^2_{\kg}}\norm{u_1}_{U^2_{\s}}\norm{u_2}_{V^2_{\s}}\\
  \end{align*}
  where   
  \begin{align*}
    \left( p,q \right) = \left( \frac{2}{\epsilon}, \frac{2}{1-\epsilon}  \right),
    \,
    \left( p_1,q_1 \right) = \left(2,\frac{10}{3} +\epsilon\right) ,
    \,\text{and}\,
    \left( p_2,q_2 \right) = \left( \frac{2}{1-\epsilon}, \left( \frac{1}{2} +
    \frac{\epsilon}{2} -\frac{ 3 }{ 10+3\epsilon } \right)^{-1} \right).
  \end{align*} 
  \emph{Case IV.2}: $ k<10 $. We use the bilinear restriction estimate (\ref{eq:candybilinear 2})
  \begin{align*}
    \lvert I \rvert  
      \lesssim  \lVert N u_1 \rVert  _{ L^{ \frac{8}{5} } L^{ \frac{3}{2} } } 
      \lVert  u_2 \rVert _{ L^{ \frac{8}{3} }  L^{ 3 }  } 
      \lesssim 2^{\frac{k}{12}-\frac{k_1}{3} -\frac{k_2}{4}  } \lVert N \rVert  _{ U_{\kg}^{ 2 }  }  \lVert u_1 \rVert _{ U_{ \s}^2} \lVert u_2 \rVert_{ V_{\s}^2 }.
   \end{align*} 
  which is acceptable since $ k_2 \sim k_1 \gtrsim c $ and $ 2^{ \frac{k}{12}  } \lesssim 1 $.
\end{proof}

\subsection{Trilinear estimates --  Klein-Gordon nonlinearity}
\begin{prop}\label{trilinearKG}
  Let $ \epsilon> 0 $ be sufficiently small, and let $N$, $u_1$ and $u_2$ be functions localized at frequency $\sim 2^k$, $2^{k_1}$ and $2^{k_2}$.
  Then,
  \begin{align}\label{eq:trilinearKG}
    \left\lvert \int_{\R^{1+3}} 
    \langle D \rangle  ^{ -1 }( u_{1}\bar u_{2}) N \dif x \dif t\right\rvert
    &\lesssim  \langle 2^{ k }  \rangle ^{ -1 } \lVert N \rVert_{V_{ \kg }^{ 2 } } \lVert u_{1} \rVert _{
    U_{ \s }^{ 2 } }  \lVert u_2 \rVert  _{ U_{ \s }^{ 2 }  }.
  \end{align} 
  When $ k < -10$ the same estimate holds when $u = \tilde P_{< -10} u$; similarly for $ k_1 $ and $ k_2 $.
\end{prop}

\begin{proof}[Proof of Proposition (\ref{trilinearKG})]
  It is enough to show that
  \begin{align}
    \left\lvert \int_{\R^{1+3}} 
    u_{1}\bar u_{2} N \dif x \dif t\right\rvert
    &\lesssim  \lVert N \rVert_{V_{ \kg }^{ 2 } } \lVert u_{1} \rVert _{ U_{ \s }^{
    2 } }  \lVert u_2 \rVert  _{ U_{ \s }^{ 2 }  }.
  \end{align} 
  We organize the proof in cases depending on the relationship between $ k_1 $ and $ k_2 $.
  Fix an absolute frequency cutoff $ c $, which we set at $ c = -3 $.

  %%%%%%%%%%%%%%%%%%%%%%% CASE 1' %%%%%%%%%%%%%%%%%%%%%%%
  \textbf{Case I}: $ \max \left\{ k, k_1, k_2 \right\} < c $.
  Since  we only use $ V^{ 2 }  $-admissible norms in (\ref{eq:ll}), this case follows with essentially the same proof.

  We shall assume for the remaining of this proof that $ \max\left\{ k ,k_1, k_2 \right\} \gtrsim
  c$ and that $ \med \left\{ k, k_1, k_2 \right\} \sim \max \left\{ k, k_1, k_2
  \right\}$.
  Furthermore, since (\ref{eq:trilinearKG}) is symmetric with respect to $ k_1 $ and $ k_2
  $, we may assume that $ k_1\geq k_2 $.

  %%%%%%%%%%%%%%%%%%%%%%% CASE 2' %%%%%%%%%%%%%%%%%%%%%%%
  \textbf{Case II}: $ \lvert k_1-k_2 \rvert \leq 10 $.
  We use the extended radial Strichartz range to obtain
  \begin{align*}
    \lvert I  \rvert 
    &\lesssim \lVert N \rVert  _{ L^{ \infty } L^{ 2 }  } 
    \lVert u_1 \rVert  _{ L^{ 2 } L^{ 4}  } \lVert u_2 \rVert_{ L^{ 2 } L^{ 4 }  } 
    \lesssim 2^{ -\frac{k_1}{2}  }  \lVert N \rVert  _{ V_{ \kg }^{ 2 }  }
    \lVert u_1 \rVert  _{ U_{ \s }^{ 2 }  } 
    \lVert u_2 \rVert  _{ U_{ \s }^{ 2 } },
  \end{align*} 
  which is acceptable since $ k_1\gtrsim c $.

  %%%%%%%%%%%%%%%%%%%%%%% CASE 3' %%%%%%%%%%%%%%%%%%%%%%%
  \textbf{Case III}: $ k_1 \succ k_2 $.
  \emph{Case III.1}:
  If $ k_1\succ 1 $, by  Lemma \ref{lemma: kg nonresonance} $I( \textbf{k})$ only has nontrivial
  contributions when $j\gtrsim 2k_1$, and as such we decompose
  \[
    I \defeq \int Nu_1 u_2\dif x\dif t = I_0 + I_1 + I_2,
  \] 
  where
  \begin{align*}
    &I_0 = \sum_{j\gtrsim 2k_1} \int Q_{ j} N Q_{\leq j}u_1 Q_{\leq j}u_{ 2}\dif x \dif t, \\
    &I_1 = \sum_{j\gtrsim 2k_1} \int Q_{< j} N Q_{ j}u_1 Q_{< j}u_{ 2}\dif x \dif t, \\
    &I_2 = \sum_{j\gtrsim 2k_1} \int Q_{< j} N Q_{< j}u_1 Q_{ j}u_{2}\dif x \dif t.
  \end{align*}

  Estimating each interaction directly, we get
  \begin{align*}
    \lvert I_{2} \rvert 
    &\lesssim \sum_{j\gtrsim 2k_1} \lVert Q_{ j } u_{2} \rVert_{L^{2}_{t,x}} \lVert Q_{\leq j}N
    \rVert _{L^{\infty}_{t}L^{3}_{x}} \lVert Q_{\leq j} u_{ 1} \rVert_{L^{2}_{t}L_{x}^{6}} \\
     &\lesssim \sum_{j\gtrsim 2k_1}2^{-\frac{j}{2}} \lVert Q_{ j } u_{2} \rVert_{\dot
     X^{0,\frac{1}{2}}}
       \lVert Q_{\leq j}N \rVert _{L^{\infty} L^{3}} \lVert Q_{\leq j}u_{1} \rVert
       _{L^{2}L^{6}}\\
     &\lesssim  \sum_{j\gtrsim 2k_1}2^{ \frac{1}{2} (k-j) } \lVert Q_{ j } u_{2} \rVert_{\dot
     X^{0,\frac{1}{2}}}
       \lVert Q_{\leq j}N \rVert _{L^{\infty}L^{2}} \lVert Q_{\leq j}u_{1} \rVert
       _{L^{2}L^{6}}\\
    &\lesssim 2^{ \frac{1}{2} (k-2k_1) } \lVert N \rVert _{ V_{ \kg }^{ 2 }  }\lVert u_{1}
    \rVert _{ U_{ \s }^{ 2 } } \lVert u_2 \rVert _{ U_{ \s }^{ 2 } }.
  \end{align*}
  The estimates for $I_0$ and $I_1$ follow analogously.

  \emph{Case III.2}: If $ \lvert k_1-1 \rvert  \leq 10$, we use the bilinear restriction estimate
  (\ref{eq:candybilinear 3})
  \begin{align*}
    \lvert I \rvert 
    &\lesssim \lVert u_1 u_2 \rVert_{ L^{ \frac{8}{5} } L^{ \frac{3}{2} } }
    \lVert N \rVert  _{ L^{ \frac{8}{3}  } L^{ 3 } } 
    \lesssim 2^{\frac{k_2}{12} -\frac{k}{4}  } \lVert N \rVert  _{ V_{ \kg }^{ 2 } }
    \lVert u_1 \rVert_{ U_{ \s }^{ 2 }  } \lVert u_2 \rVert  _{ U_{ \s }^{ 2 }  },
  \end{align*} 
  which is acceptable because $ k \sim 1 $.
\end{proof}

We can now determine the regularity exponents $ s $ and $ r $ for which
(\ref{eq:finalSchrodinger}) and (\ref{eq:finalKleinGordon}) hold.
For $k,k_1,k_2\geq -10 $, we can combine (\ref{eq:finalSchrodinger}) and (\ref{eq:finalKleinGordon})
into the more compact estimate
\begin{align*}
  \left\lvert \int_{ \R^{ 1+3 }  } Nu_1 u_2\dif x\dif t \right\rvert 
  &\lesssim 
  \min\left\{
    2^{ rk+sk_1-sk_2 } 
    \lVert N \rVert  _{ U_{ \kg }^{ 2 } }
    \lVert u_1 \rVert  _{ U_{ \s }^{ 2 }  } 
    \lVert u_2 \rVert  _{ V_{ \s }^{ 2 }  },
    \right.\\
  &\hspace{13em}
  \left.
    2^{ (1-r)k + sk_1 + sk_2} 
    \lVert N \rVert  _{ V_{ \kg }^{ 2 } }
    \lVert u_1 \rVert  _{ U_{ \s }^{ 2 }  } 
    \lVert u_2 \rVert  _{ U_{ \s }^{ 2 }  }
  \right\},
\end{align*} 
which is compatible with Propositions (\ref{trilinearS}) and (\ref{trilinearKG}) provided we
choose $ s $ and $ r $ satisfying
\begin{align*}
  2^{ \left( -\frac{1}{2} + \frac{\epsilon}{2}  \right)k } \lesssim 2^{ rk+sk_1-sk_2 }
  \quad\text{and}\quad
  2^{ -k } \lesssim 2^{ (1-r)k + sk_1 + sk_2}.
\end{align*} 
This forces $ s \geq 0 $ and $ \left( s-\frac{1}{2} + \frac{\epsilon}{2}  \right) \leq r\leq
\left( s+2 \right) $.

By shrinking the range on $ s $, we can sum over different frequencies and upgrade the frequency localized
estimates to functions with arbitrary frequency support.

\begin{lemma}\label{G}
  Let  $ s \geq 0 $ and let $ \left( r -\frac{1}{2} + \frac{\epsilon}{2} \right)  < s <
  \left( r + 2 \right)  $ be as above.
  Then, 
  \begin{align*}
    \left\lvert \int N u_1u_2\dif x\dif t \right\rvert 
    &\lesssim 
      \min\left\{
        \lVert N \rVert    _{ Z_{ \kg}^{ r } }
        \lVert u_1 \rVert  _{ Z_{ \s }^{s}  } 
        \lVert u_2 \rVert  _{ \ell^{ 2 } V_{ \s } ^{ 2,-s }   },
        \lVert N \rVert    _{ \ell^{ 2 } V^{ 2,-r }_{ \kg} }
        \lVert u_1 \rVert  _{ Z_{ \s }^{s}  } 
        \lVert u_2 \rVert  _{ Z_{ \s }^{s}  }
    \right\}.
  \end{align*} 
\end{lemma}

\begin{proof}
  For frequencies above $-10 $, the claim follows immediately from Propositions \ref{trilinearS} and \ref{trilinearKG} and the
  observation that, by Hölder's and Young's convolution inequalities,
  \begin{align*}
    \sum_{ \substack{k,k_1,k_2\in\N \\ |\max\left\{ k,k_1,k_2 \right\} -\med\left\{
    k,k_1,k_2 \right\}|\leq 10} } 2^{ -\delta \min\left\{ k,k_1,k_2 \right\}  }x_{ k } y_{
    k_1 } z_{ k_2 }  
    &\lesssim_{ \delta }  \lVert x \rVert_{ \ell^{ 2 }  } \lVert y \rVert_{ \ell^{ 2 }  } \lVert z
    \rVert  _{ \ell^{ 2 }  },  
  \end{align*} 
  for any $ \delta > 0 $ and  all $ x,y,z\in\ell^{ 2 } (\N) $.
  For low frequencies, the result follows immediately from our choice of projections 
  $ P_{ < -10 } $ which avoids summability issues.
\end{proof}

%\begin{coro}\label{coro:extragain}
%  Let $ k,k_1,k_2\geq -10 $ and  $u_{j} = \tilde P_{k_{j}}u_{j}$, $ j = 1,2 $, and $N
%  = \tilde P_{k}N$ be frequency localized functions.
%  Then,
%  \begin{align*}
%    \left\lvert \int_{\R^{1+3}} N_{k}u_{1}u_{2}\dif x\dif t \right\rvert
%    &\lesssim  2^{ -\epsilon \min(k,k_1,k_2)} 
%      \min\left(\lVert N_{k} \rVert_{U_{ \kg }^{ 2,s } } \lVert u_{1} \rVert _{ U_{ \s,k_1 }^{ 2,s }
%      } \lVert u_{2} \rVert _{V_{ \s }^{ 2,-s } },\right.\\
%      &\hspace{14em} \left.\lVert N_{k} \rVert_{V_{ \kg,k }^{ 2,-s } } \lVert u_{1} \rVert
%      _{ U_{ \s,k_1 }^{ 2,s } } \lVert u_{2} \rVert _{U_{ \s }^{ 2,s } } \right) 
%  \end{align*} 
%  for any $ 0 < \epsilon \leq s $.
%\end{coro}
%

\subsection{Proof of Theorem \ref{mainthm}}\label{mainproof}

This argument is similar to the one in \cite{BH-dkg}, except we use different resolution
spaces.

Fix $ \epsilon > 0 $.
For simplicity, we will write the proof for $ s = 0 $ and $ r = -\frac{1}{2} +\epsilon $;
the general case follows analogously.
We will construct a solution
\begin{align*}
  \left( u,N \right) 
  &\in  Z_{ \s }^{ 0 } \times Z_{ \kg }^{ -\frac{1}{2} + \epsilon }
\end{align*} 
of the system (\ref{eq:modifiedkgs}) in integral form
\begin{nalign}\label{eq:integralform}
  u(t) &= e^{ -it\Delta } u_0 + \frac{1}{2} \int_{ 0 }^{ t } e^{ -i(t-s)\Delta } \left( uN +
  u\bar N\right) \dif s\\
    N(t) &= e^{ it\langle D \rangle   } N_0 + \int_{ 0 }^{ t } e^{ i\left( t-s \right)
    \langle  D \rangle } \langle D \rangle  ^{ -1 } \lvert u \rvert  ^2\dif s
\end{nalign} 
provided that the initial data is radial and satisfies
\begin{align*}
  &\left\lVert \left( u_0, N_0 \right)  \right\rVert_{ L_{ x }^{ 2 } \times H_{ x }^{
  -\frac{1}{2}  + \epsilon}  } \leq \delta,
\end{align*} 
for sufficiently small $ \delta $.
Let $ T(u,N) $ denote the right-hand side of (\ref{eq:integralform}).
By Lemmas \ref{inhomogeneous} and \ref{G} we conclude that
\begin{align*}
  \lVert T(u,N) \rVert _{ Z_{\s}^{0} \times Z_{\kg}^{-\frac{1}{2} +\epsilon} }
   &\lesssim\delta + \lVert u \rVert  _{ Z_{\s}^{0}  } \lVert N \rVert_{ Z_{\kg}^{-\frac{1}{2} +\epsilon} } + \lVert u \rVert  _{ Z_{\s}^{0} }^2
    \lesssim \delta + \left\lVert (u,N) \right\rVert _{ Z_{\s}^{0}\times Z_{\kg}^{-\frac{1}{2} +\epsilon} } ^2,
    % bounds for differences below
   % \\
   % \lVert T(u_1,N_1) - T(u_2,N_2) \rVert_{ Z_{\s}^{0} \times Z_{\kg}^{-\frac{1}{2} +\epsilon}}
   %   &\lesssim
   %   \left( \lVert (u_1,N_1) \rVert_{ Z_{\s}^{0} \times  Z_{\kg}^{-\frac{1}{2} +\epsilon} }
   %   \right. + \left. \lVert (u_2,N_2) \rVert_{ Z_{\s}^{0} \times  Z_{\kg}^{0} } \right) \\
   %   &\hspace{13em}\times\lVert (u_1-u_2,N_1-N_2) \rVert_{ Z_{\s}^{0} \times  Z_{\kg}^{-\frac{1}{2} +\epsilon} }.
\end{align*} 
and similar bounds for differences.
Therefore,  we can use the contraction principle in a small ball in this space to obtain a
unique solution that depends continuously on the initial data. 
%Moreover, if the initial data $ (u_0,N_0)\in H^{s} \times H^{ s }  $, the solution
%satisfies $ (u,N)\in C_{ t }^{ 1 } H_{ x }^{ s }\times C_{ t }^{ 1 } H_{ x } ^{ s } $.

To prove scattering, it is enough to show that the integrals
\begin{nalign}\label{eq:tails}
  \int_{ 0 }^{ t } e^{ is\Delta } \left( uN + u\overline N\right) \dif s
  \quad
  \text{and}
  \quad
  \int_{ 0 }^{ t } e^{i s \langle  D \rangle } \langle D \rangle  ^{ -1 } \lvert u \rvert
  ^2\dif s
\end{nalign}
converge in $ L^{ 2 }  \times H^{ -\frac{1}{2} +\epsilon }  $ as $ t\to\infty $.
Indeed, if this is the case, denote the limits of the integrals in
(\ref{eq:tails}) by $ \tilde u $ and $ \tilde N $, respectively.
By Duhamel's formula (\ref{eq:integralform}),
\begin{align*}
  u(t)-e^{ -it\Delta }(u_0+\tilde u)
  &= -e^{ -it\Delta } \int_{ t }^{ \infty } e^{ is\Delta } \left( uN+ u\overline N \right)
  \dif s,\\
  N(t) - e^{ it\langle D \rangle   } (N_0+\tilde N)
  &= e^{ it\langle D \rangle   } \int_{ t }^{ \infty } e^{ -is \langle D \rangle  }\langle D
  \rangle  ^{ -1 } \lvert u \rvert  ^2\dif s,
\end{align*}
and since the operators $ e^{ it\Delta }$ and $ e^{ -it\langle D \rangle   }  $ are unitary, 
we conclude that
\begin{align*}
  \lim_{ t\to\infty } \lVert u(t)-e^{ -it\Delta } (u_0+\tilde u) \rVert _{ L^{ 2 }_{ x }  }
    = \lim_{ t\to \infty } \lVert N(t)-e^{ it\langle D \rangle  (N_0+\tilde N) }  \rVert _{
    H^{ -\frac{1}{2} + \epsilon }_{ x }  } = 0.
\end{align*}

Now, to prove the convergence of (\ref{eq:tails}), by Proposition \ref{prop:ellp} we have
\begin{align*}
  \int_{ 0 }^{ t } e^{- i(t-s)\Delta } \left( uN +
  u\overline N\right) \dif s \in  Z_{\s}^{ 0 } \subset V^{ 2}_{\s}
  \quad
  \text{and}
  \quad
  \int_{ 0 }^{ t } e^{i(t-s)\langle  D \rangle } \langle D \rangle  ^{ -1 } \lvert u \rvert
  ^2\dif s\in Z_{\kg}^{ -\frac{1}{2} + \epsilon }.
\end{align*}
By (\ref{eq:scattering}), scattering follows.

\subsection*{Acknowledgments}
We thank Ioan Bejenaru and Federico Pasqualotto for many helpful discussions and suggestions, including feedback on previous drafts of the paper.

\subsection*{Data Availability Statement}
No datasets were generated or analyzed during the current work. 

\subsection*{Conflict of Interest Statement}
The authors have no relevant financial or non-financial interests to disclose.

%\section{Zakharov}
%\input{zakharov}

%\section{Resonance analysis}\label{section 6}
%\input{6 resonance analysis}

%\section{Old stuff}
%\input{7 old stuff}

%\bibliographystyle{plain}
%\bibliography{./refs.bib}

\printbibliography

\end{document}